\documentclass[reqno]{amsart} 
\makeatletter
\@namedef{subjclassname@2020}{%
  \textup{2020} Mathematics Subject Classification}
\makeatother

\usepackage[utf8]{inputenc} 
\usepackage[english]{babel}
\usepackage{mathabx}
\usepackage{amstext}
\usepackage{amssymb}
\usepackage{amsthm}
\usepackage{amsmath}
\usepackage{amstext}
\usepackage{amsfonts}

\usepackage[foot]{amsaddr} 
\usepackage{cite} 

\usepackage[usenames,dvipsnames]{xcolor} 
\definecolor{dkblue}{RGB}{1,31,91} 
\usepackage[colorlinks=true, pdfstartview=FitV, linkcolor=dkblue, citecolor=dkblue, urlcolor=dkblue]{hyperref}

\theoremstyle{definition}
\newtheorem{theorem}{Theorem}

\newtheorem{lemma}[theorem]{Lemma}
\newtheorem{proposition}[theorem]{Proposition}
\newtheorem{remark}[theorem]{Remark}
\newtheorem{definition}[theorem]{Definition}

\numberwithin{equation}{section}
\numberwithin{theorem}{section}

\def\C{\mathbb C}
\def\R{\mathbb R}
\def\Z{\mathbb Z}
\def\N{\mathbb N}
\def\T{\mathbb T}
\def\supp{\mathop{\text{supp}}}

\def\grad{\triangledown}
\def\sgn{\text{\normalfont sgn}}

\begin{document}


%
\title[The Second Iterate of the Muskat Equation in Supercritical Spaces]{The Second Iterate of the Muskat Equation in Supercritical Spaces}

\author[E. Paduro]{Esteban Paduro$^{\dagger}$}
\address{$^\dagger$Instituto de Ingenier\'ia Matem\'atica y Computacional, Facultad de Matem\'aticas, Pontificia Universidad Cat\'olica de Chile, Santiago, Chile.  \href{mailto:esteban.paduro@mat.uc.cl}{esteban.paduro@mat.uc.cl} \newline(\href{https://orcid.org/0000-0003-1769-0055}{https://orcid.org/0000-0003-1769-0055})
}
\address{$^\dagger$Department of Mathematics, University of Pennsylvania, Philadelphia, PA 19104, USA.} 
\thanks{$^\dagger$Partially supported by National Agency for Research and Development (ANID)/Becas Chile/Becas de doctorado en el extranjero 2015-72160564 and the ANID Millennium Science Initiative Program trough Millennium Nucleus for Applied Control and Inverse Problems NCN19-161}

\begin{abstract}
The ill-posedness for the Muskat problem in spaces that are supercritical with respect to the scaling is studied. The main result of the paper establishes that for a sequence of  approximations of the Muskat equation obtained via Taylor expansion, their corresponding second Picard's iterate is discontinuous around the origin in a certain family of supercritical spaces approaching a critical space. 
\end{abstract}

\thispagestyle{empty}
\maketitle

\section{Introduction}

\subsection{Description of the model}

The Muskat equation describes the interface between two incompressible immiscible fluids with different densities in a porous media. The evolution problem can be described as a transport equation where the velocity field is incompressible and the evolution of velocity of the fluid for a porous media is given by the Darcy's law
$$\vec{u} =  - \grad p -  \rho \vec{e}_{n}, $$
where $\vec{u}$ is the velocity, $p$ is the pressure, $\rho$ is the density and $\vec{e}_n$ is the last vector in the canonical base. This paper focuses in the 2D case in the situation in which we have two fluids of constants densities with same viscosity, the denser fluid is at the bottom and the surface tension is ignored. The main case of interest is where the interface is formed between water and oil \cite{muskat1934}. To write an equation for the interface we consider the regime where it can be described by the graph of a function 
\begin{equation*}
\Sigma(t) = \left\{(x_1, f(x_1,t))\in \R^2:  x_1\in\R\right\},
\end{equation*}
and consequently the density can be written as
\begin{equation*}
\rho(x,t) = \begin{cases} 	
\rho_1 & ,  x\in \Omega_1(t) := \{(x_1, x_2)\in \R^2: x_2 > f(x_1)\}\\ 
\rho_2 & ,  x\in \Omega_2(t) := \R^2 \setminus \Omega_1(t). 
\end{cases} 
\end{equation*}
Under these assumptions the initial value problem for the evolution of the interface is given by (see \cite{Constantin2013})
\begin{equation}\label{muskat_non_periodic}
\begin{cases}
{\displaystyle \partial_t f + \frac{\rho_2-\rho_1}{2}\Lambda f = -\frac{\rho_2-\rho_1}{2\pi}p.v.\int_\R \frac{\partial_x \delta_\alpha f(x)}{\alpha} \frac{(\delta_\alpha f(x))^2}{\alpha^2 + (\delta_\alpha f(x))^2}d\alpha},\\
f(x,0) = f_0(x), \quad x \in \R, 
\end{cases}
\end{equation}
where $\rho_2 > \rho_1$, $\hat{f}(\xi) = \int_\R e^{-2\pi i x \xi} f(x) dx $, $\mathcal{F}(\Lambda f) = 2\pi |\xi| \hat{f}$, $\delta_\alpha f(x) = f(x) - f(x-\alpha)$ and the principal value is taken at zero or infinity if needed. Additionally, without loss of generality it is assumed that $\frac{\rho_2-\rho_1}{2}=1$. 

To study the well-posedness of problem \eqref{muskat_non_periodic} it is useful to consider the following family of homogeneous Besov-type spaces. Note that because \eqref{muskat_non_periodic} is invariant under addition of constants it is convenient to study the problem in homogeneous spaces.
\begin{definition}[The $\dot{\mathcal{F}}^{s,p}_q$ norm]
For $k\in \Z$, we consider the annulus $C_k = \{x\in \R: 2^{k}\leq |x| < 2^{k+1}\}$ and for $s\in\R$, $p\geq 1$, $q\geq 1$ we consider the norm
\begin{equation}\label{definition_norm_non_periodic}
\|f\|_{\dot{\mathcal{F}}^{s,p}_q} = \left(\sum_{k\in \Z}\left(\int_{C_k} |\xi|^{sp} |\hat{f}|^p d\xi\right)^{q/p}\right)^{1/q}, \hspace{0.5cm}f\in C_c^{\infty}(\R),
\end{equation}
and the space $\dot{\mathcal{F}}^{s,p}_q(\Omega)$ is defined as the closure of $C_c^{\infty}(\Omega)$ with respect to the $\dot{\mathcal{F}}^{s,p}_q(\Omega)$ norm. A notation that is sometimes used is $\dot{\mathcal{F}}^{s,p} := \dot{\mathcal{F}}^{s,p}_1$ and $\dot{\mathcal{F}}^{s} := \dot{\mathcal{F}}^{s,1}_1$. Note that this family of spaces contain the Wiener algebra  $\mathbb{A}^1 = \dot{\mathcal{F}}^{1,1}_1$.
\end{definition}

The Muskat equation satisfy the following scaling property: let $f(x,t)$ be a solution of \eqref{muskat_non_periodic}, then $f_\lambda(x,t) = \frac{1}{\lambda}f(\lambda x, \lambda t)$ is also a solution of \eqref{muskat_non_periodic} for any $\lambda >0$. Spaces whose norm is preserved under this scaling are called critical spaces. When considering a family of spaces, we say that spaces that are more regular than the critical ones are subcritical and the ones that are less regular are called supercritical. Relevant critical spaces encountered while studying the Muskat problem in $\R^{d+1}$  ($d$-dimensional interface) are the homogeneous spaces $\dot{C}^1$, $\dot{\mathcal{F}}^{1+d \frac{p-1}{p},p}_q$, $ \dot{W}^{1+d/p,p}$  and $\dot{B}^{1+d/p}_{p,q}$.

\subsection{Main result}

Consider a sequence of  approximations of the Muskat equation \eqref{muskat_non_periodic} obtained by considering the Taylor expansion of order $\ell$ of the nonlinear term.
\begin{definition}
A function $f$ is said to be the solution of the truncation or order $\ell$ of the Muskat problem if 
\begin{equation}\label{muskat_problem_truncated}
\begin{cases}
\partial_t f + \Lambda f  =  \sum_{k= 1}^\ell T_k f &,  (x,t) \in \R \times [0,T],\\
f(0) = f_0 &,  x\in \R,
\end{cases}
\end{equation}
where the function $T_k$ is given by
\begin{align}
T_{k} f &=  (-1)^k\frac{1}{\pi}p.v. \int_\R \frac{\partial_x \delta_\alpha f(x)}{\alpha} \left(\frac{\delta_\alpha f(x)}{\alpha}\right)^{2k} d\alpha \nonumber \\
&= \frac{(-1)^k}{\pi(2k+1)} p.v. \int_\R \partial_x \left(\frac{\delta_\alpha f(x)}{\alpha}\right)^{2k+1} d\alpha. \label{terms_tk_expasion_taylor_muskat_real_line}
\end{align}
\end{definition}
To study the ill posedness we follow the strategy in \cite{Germain2008,Iwabuchi2016} and prove that the solution map of the second Picard's iterate of  \eqref{muskat_problem_truncated} is discontinuous at the origin. In many situations the Picard's iteration it is expected to converge to a solution of the problem, but in the case of supercritical spaces this is a difficult question in general, this is why in this paper we only focus our attention to the evolution of the second Picard's iteration for some highly oscillatory initial data. 

The Picard's iterations $\{f^{(n)}\}_{n \geq 0}$ of equation \eqref{muskat_problem_truncated} are defined in the following way: set $f^{(0)}  = 0$ and for $n\geq 1$ define recursively
\begin{equation*}
\partial_t f^{(n)} + \Lambda f^{(n)} = \sum_{k= 1}^\ell T_k f^{(n-1)}, ~~f^{(n)}(0) = f_0,
\end{equation*}
from this definition the first two Picard's iterations are given by
\begin{equation*}
\partial_t f^{(1)} + \Lambda f^{(1)}  = 0 ,~~ f^{(1)}(0) = f_0 \Rightarrow ~ f^{(1)} = e^{-t\Lambda} f_0,
\end{equation*}
\begin{equation*}
\partial_t f^{(2)} + \Lambda f^{(2)} = \sum_{k= 1}^\ell T_k e^{-t\Lambda} f_0, ~ f^{(2)}(0) = f_0.
\end{equation*}
Our goal is to show that by choosing an appropriate initial condition we can make the term $f^{(2)}$ arbitrarily large compared with the initial data after an arbitrarily short time. 

\begin{definition}
Let $\ell \in \N$, $p\geq 1$, $q\geq 1$, $T>0$. Given $\varphi\in \dot{\mathcal{F}}^{\frac{2\ell-1}{2\ell+1},p}_q(\R)$ the second Picard's iterate of \eqref{muskat_problem_truncated} is a function  $f \in C([0,T];\dot{\mathcal{F}}^{\frac{2\ell-1}{2\ell+1},p}_q(\R))$ that satisfy 
\begin{equation}\label{muskat_problem_truncated_second_iteration}
\begin{cases}
\partial_t f + \Lambda f = \sum_{k= 1}^\ell T_k e^{-t\Lambda} \varphi&, (x,t)\in\R\times [0,T],  \\
f(x,0) = \varphi(x) &, x \in \R,
\end{cases}
\end{equation}
in the weak sense.
\end{definition}
We can now state the main result of this paper.

\begin{theorem}[Norm inflation for truncated system]\label{thm_norm_inflation_finite_expansion_muskat} Let $\ell\in \N$, $T>0$, $R > 0$, $p\geq 1$,  $q> 2\ell+1$. Then there exists some $\tilde{t}\in (0,T)$, and a function $\varphi_0 \in \dot{\mathcal{F}}^{\frac{2\ell-1}{2\ell+1},p}_q(\R)$ such that the solution $f \in C([0,T];\dot{\mathcal{F}}^{\frac{2\ell-1}{2\ell+1},p}_q(\R))$ of \eqref{muskat_problem_truncated_second_iteration} with  $\varphi = \varphi_0$ satisfy
\begin{equation*}
\|f(0)\|_{\dot{\mathcal{F}}^{\frac{2\ell-1}{2\ell+1},p}_q} < 1/R \text{ ~~and~~ }\|f(\tilde{t})\|_{\dot{\mathcal{F}}^{\frac{2\ell-1}{2\ell+1},p}_q} >R.
\end{equation*}
\end{theorem}

To make precise what we mean by discontinuity of the solution map for the approximation of the Muskat problem given by equation \eqref{muskat_problem_truncated_second_iteration} at the origin we can consider the operator
\begin{equation}\label{solution_map}
L : \dot{\mathcal{F}}^{\frac{2\ell-1}{2\ell+1},p}_q(\R) \to C([0,T];\dot{\mathcal{F}}^{\frac{2\ell-1}{2\ell+1},p}_q(\R)),
\end{equation}
that takes a function $\varphi_0 \in \dot{\mathcal{F}}^{\frac{2\ell-1}{2\ell+1},p}_q(\R)$ and return the solution 
$$f\in C([0,T];\dot{\mathcal{F}}^{\frac{2\ell-1}{2\ell+1},p}_q(\R)),$$ 
of \eqref{muskat_problem_truncated_second_iteration} with  $f(t,0) = \varphi_0$. Theorem \ref{thm_norm_inflation_finite_expansion_muskat} tell us that given some arbitrarily small $T>0$ it is possible to find a decreasing sequence of times and function $\{(t_N, \varphi_N)\}_{N=1}^\infty$ with $t_N< T$, such that the sequence $\{f_N  = L \varphi_N\}_{N=1}^\infty$ satisfies
\begin{equation*}
\|\varphi_N\|_{\dot{\mathcal{F}}^{\frac{2\ell-1}{2\ell+1},p}_q} \leq \frac{1}{N} \text{ and } \|f_N(t_N)\|_{\dot{\mathcal{F}}^{\frac{2\ell-1}{2\ell+1},p}_q} > N,
\end{equation*}
which implies that the solution map $L : \dot{\mathcal{F}}^{\frac{2\ell-1}{2\ell+1},p}_q \to C([0,T];\dot{\mathcal{F}}^{\frac{2\ell-1}{2\ell+1},p}_q)$ is not continuous around $\varphi = 0$ in $\dot{\mathcal{F}}^{\frac{2\ell-1}{2\ell+1},p}_q(\R)$, for $p\geq 1$ , $q > 2\ell+1$.  In particular we can look at the sequence of spaces $\dot{\mathcal{F}}^{\frac{2\ell-1}{2\ell+1},1}_{2\ell+1+\varepsilon}(\R)$ as a family approaching the critical space $\dot{\mathcal{F}}^{1,1}_\infty$ as we increase the order of the approximation i.e.,  $\ell \to \infty$. 

\subsection{Choice of initial data}\label{subsection:initial_data}
A crucial ingredient in the proof of Theorem \ref{thm_norm_inflation_finite_expansion_muskat} is the choice of initial data and its dependence on various parameters. Given $\ell \in \N$, $p \geq 1$, $q> 2\ell+1$, $T > 0$,  
$0<\varepsilon< \frac{q}{2\ell + 1} - 1$, and $M >2\ell+2$ we want to construct a sequence of functions $\varphi^{(N)} \in \dot{\mathcal{F}}^{\frac{2\ell-1}{2\ell+1},p}_q (\Omega)$, $N\in \N$ with small norm such that the solution of \eqref{muskat_problem_truncated_second_iteration} becomes large after a short time $\tilde{t}<T$. The structure of the initial data considered in this work is inspired by the works of Bourgain-Pavlovic \cite{Bourgain2008} and Iwabuchi-Ogawa \cite{Iwabuchi2016}. For each $N\in \N$ we define the real-valued function $\varphi^{(N)}\in L^2(\R)$ by
\begin{equation}\label{intial_condition_higher_order_muskat}
\hat{\varphi}^{(N)}(\xi) = \sum\limits_{j=N}^{(1+\delta) N} \gamma_j \left(P_{k_j}(\xi) + P_{2 \ell k_j + M}(\xi)\right), ~~\xi \in \R,
\end{equation}
where $P_A(\xi)  = \chi(\xi-A) + \chi(\xi+A)$ for $A\in \R$ and $\chi(\xi)$ denotes the characteristic function of the interval $[-1,1]$. Another notation we will use later is $\chi_A(\xi) = \chi(\xi-A)$. 
The key of this construction is the choice of the sequences $\{k_j\}_j$ and $\{\gamma_j\}_j$. Here $\{k_j\}_{j\geq 0}$ is a sequence of positive integers that grow very fast and $\{\gamma_j\}_{j\geq 0}$ a sequence of positive numbers that depend on $\{k_j\}_{j\geq 0}$. The precise growth requirements are given by the following conditions
\begin{align}
(a)& ~k_{j+1} > \ell^2 k_j + M \text{ if } \ell > 1 \text{ and }~k_{j+1} > 2 k_j + M \text{ for }\ell = 1,\label{growth_condition_kn_1} \\
(b)& ~\sum_{j=N}^{(1+\delta)N} \frac{1}{j^{\frac{2\ell-1}{2\ell+1} \frac{1+\varepsilon}{q}}} < \frac{1}{N}  k_N^{\frac{1}{2\ell+1}}, \label{growth_condition_kn_2}\\
(c)& ~ (2\ell+1)M < k_0/2, \label{growth_condition_kn_3}\\
(d)& ~\gamma_j = \frac{1}{ j^{\frac{1+\varepsilon}{q}}} k_j^{-\frac{2\ell-1}{2\ell+1}} \text{ for } j\in\N. \label{definition_gammaj}
\end{align}

\subsection{Summary of Known Results}

In the Rayleigh-Taylor (RT) unstable case ($\rho_1 > \rho_2$) the problem is known to be ill-posed in the Sobolev spaces $H^s$ for $s > 3/2$ in 2D and 3D \cite{Cordoba2007,Cordoba2008}. 

The parameterization of the problem as the graph of a function can hide some nuances of the equation. In a series of papers \cite{Castro2011,Castro2012,Castro2013} the authors study this phenomenon in 2D and obtain that in the RT stable case ($\rho_2 > \rho_1$) some solutions can turn and then go back to the stable regime and others can turn and later stop being $C^4$.

For short time existence in the RT stable case ($\rho_1 < \rho_2$), the 2D problem without surface tension for initial data Sobolev spaces has been considered in many authors \cite{Cordoba2007,Cordoba2013,Constantin2015,Cheng2016,Matioc2017,Matioc2019,Alazard2020,Nguyen2020,Alazard2021} for initial data in $\dot{W}^{1,\infty} \cap H^{s}$ ($s\geq 3/2$) and \cite{chen2022} for initial data in $\dot{C}^1 \cap L^2$. When viscosity jump is allowed  \cite{Gancedo2017} in 2D and 3D for initial data in $L^2\cap\mathcal{F}^{1,1}$ with small $\dot{\mathcal{F}}^{1,1}$ norm and \cite{ABELS2022} for the 2D problem with initial data in the sub-critical space $W^{s,p}$ ($1+1/p<s<2$, $1<p<\infty$). In 3D without viscosity jump \cite{Cordoba2007,Cordoba2013} for initial data in Sobolev spaces $H^s$ ($s \geq 4$) and \cite{chen2022} for initial data in $\dot{C}^1 \cap L^2$. For the 2D problem with surface tension \cite{Matioc2017} for initial data in $H^{s}(\T)$ ($2<s<3$).

For global in time results in the RT stable case, the 2D problem without viscosity jump was considered in \cite{Constantin2015} for initial data in $W^{2,p}(\R)\cap W^{1,\infty}\cap L^2$ ($1<p\leq \infty$) with small slope, \cite{Patel2017} for initial data $f_0\in H^\ell$ ($\ell\geq 3$) with  $\|f_0\|_{\dot{\mathcal{F}}^{1,1}} < k_0$ (they call it medium size data as they provide an explicit lower bound for the size of the constant), \cite{Cheng2016} for initial data in $H^2$ with small $H^{3/2+\varepsilon}$ norm. \cite{Constantin2013,Constantin2016} for initial data in $H^3$ ($s\geq 3$) with medium size $\dot{\mathcal{F}}^{1,1}$ norm, \cite{Matioc2019} for initial data in $H^s$ ($3/2<s<2$) with medium size $\dot{\mathcal{F}}^{1,1}$ norm. In \cite{Gancedo2019} for the problem with bubble geometry for $\dot{\mathcal{F}}^{1,1}\cap \mathcal{F}^{0,1}$ with medium size $\dot{\mathcal{F}}^{1,1}$ norm. In 3D
\cite{Constantin2013,Constantin2016} for initial data in $H^s$ ($s\geq 4$)  with $\|\grad f_0\|_{L^\infty} < 1/3$, \cite{Cameron2020} for unbounded initial data with medium size slope and slow growth at infinity. In 2D \cite{alonso-oran2021} when permeability jump is allowed global existence for initial data $W^{1,\infty}$ with small slope.

Many of the previously mentioned results require some control on the size of the slope (because $\|u'\|_{L^{\infty}}\leq \|u\|_{\dot{\mathcal{F}}^{1,1}}$), but does not seem to be a strict requirement to obtain global in time results, for instance in \cite{Deng2017} global existence for monotone initial data with finite limits at infinity and \cite{Cameron2019} for initial data in $W^{1,\infty}$ with $(\sup f'_0)(\sup - f'_0) < 1$. \cite{Cordoba2018} for the 2D problem with initial data in $H^{5/2} \cap H^{3/2}$ and small $\dot{H}^{3/2}$ norm. \cite{gancedo_global_2022} for the 3D problem with initial data in $\dot{W}^{1,\infty}\cap \dot{H}^{2}$ with small $\dot{H}^2$ norm where the required size depend on the maximum size of the slope. \cite{Alazard2021} for the 2D problem with small initial data in $\dot{W}^{1,\infty}\cap H^{3/2}$. \cite{Nguyen2022} for small initial data in the critical space $\dot{B}^{1,\infty}_1$. For the one phase problem ($\rho_1 = 0$) global existence of viscosity solutions for initial data in $W^{1,\infty}(\mathbb{T})$ is obtained in \cite{dong2021}. In \cite{Alazard2021b} global existence for the 2D problem in Sobolev-type spaces with logarithmic weight $\mathcal{H}^{3/2,1/3}$ where $H^{3/2+\varepsilon}\subset\mathcal{H}^{3/2,1/3} \subset H^{3/2}$, which allow non-Lipschitz initial data to be considered.

The approach of studying ill-posedness by looking at special low regularity initial data has been successfully applied to other fluid problems. For the 3D Navier-Stokes equation (NSE) the norm inflation in the critical space $\dot{B}^{-1,\infty}_\infty$ and \cite{Wang2015} for the $\dot{B}^{-1,\infty}_q$ case. In \cite{Germain2008} discontinuity at the origin for the second Picard's iterate of NSE is established in $B^{-1,\infty}_{q}$ ($q>2$, $d\geq 2$).  In \cite{Cheskidov2012} for discontinuity of the solution map in a critical space for the 3D NSE with fractional diffusion. \cite{Iwabuchi2022} discontinuity of the solution map for compressible NSE in $\dot{B}^{-1/2,2d}_\sigma$ ($1\leq \sigma \leq \infty$, $d\geq 2$). \cite{Cheskidov2010} discontinuity of the solution map at the origin in a periodic domain for Euler in $B^{s,r}_\infty$ ($s > 0$ if $r>2$ and $s>d(r/2-1)$ if $1\leq r \leq 2$) which includes the critical space $B^{1/3,3}_\infty$. \cite{Bourgain2015} for Euler with $d\geq 2$, a small perturbation of a initial data in the critical space $H^{d/2+1}$ exhibits norm inflation. \cite{Iwabuchi2016} norm inflation for a drift diffusion system in $\dot{B}^{-2+n/p,p}_\sigma$ ($2n<p \leq \infty$ and $1 \leq \sigma \leq \infty$, or $p=2d$ and $2 < \sigma \leq \infty$). \cite{Misioek2016} for 2D Euler the discontinuity of the solution map in $C^1(\R)$ and $B^{1,\infty}_1(\R)$. \cite{Li2020} norm inflation for the 2D viscous shallow water in Besov spaces. \cite{Li2020b} norm inflation for Boussinesq in Besov spaces. \cite{Tsurumi2019} discontinuity of the solution map for a stationary NSE in Besov spaces. \cite{Byers2006} for ill-possedness of the Camassa-Holm the problem is in the supercritical spaces $H^s$ ($1<s<3/2$).

\section{A norm inflation result}\label{subsection:inflation_result}

In order to prove Theorem \ref{thm_norm_inflation_finite_expansion_muskat} we will consider the following intermediate result that will give us conditions for the discontinuity of the solution map \eqref{solution_map}. Using the linearity  of \eqref{muskat_problem_truncated_second_iteration} we can consider the following decomposition 
\begin{equation}\label{muskat_problem_truncated_second_iteration_decomposition}
\begin{cases}
\partial_t f_{0} + \Lambda f_{0} = 0 &,  (x,t)\in\R\times [0,T], \\
f_0(0) = \varphi &,  x \in \R, \\
\partial_t f_{k} + \Lambda f_{k} = T_k f_0 &, \text{ for } 1\leq k \leq \ell,  (x,t)\in\R\times [0,T],  \\
f_{k}(x,0) = 0  &, \text{ for } 1\leq k \leq \ell,  x \in \R. 
\end{cases}
\end{equation}
It is easy to see that if $(f_0, \cdots, f_{\ell})$ is a solution of the system \eqref{muskat_problem_truncated_second_iteration_decomposition} then the function $f = \sum_{j=0}^\ell f_{j}$ is a solution of  \eqref{muskat_problem_truncated_second_iteration}. Another representation that will be useful later is obtained applying the Duhamel's formula to \eqref{muskat_problem_truncated_second_iteration_decomposition}
\begin{equation}\label{equation_toy_system_muskat_higher_order}
f_{j}(t) = \int_0^t e^{-(t-\tau)\Lambda} T_j(f_0) d\tau.
\end{equation}

\begin{theorem}[Discontinuity of solution map]\label{thm_inflation_higher_order_muskat}
Let $\ell\in \N$, $p\geq 1$ and $q> 2\ell+1$. Suppose that there exists $\bar{T}>0$ and $\tilde{N}\in \N$ such that for each $N \geq \tilde{N}$ there exist a function $\varphi_0^{(N)}\in \dot{\mathcal{F}}_q^{\frac{2\ell-1}{2\ell+1},p}(\R)$ such that for all $0<t\leq \bar{T}$ the solution $f_{k}^{(N)}\in C(0,\bar{T}; \dot{\mathcal{F}}^{\frac{2\ell-1}{2\ell+1},p}_q(\R))$, $k = 0, \cdots, \ell$ of \eqref{muskat_problem_truncated_second_iteration_decomposition} with $\varphi = \varphi_0^{(N)}$  satisfy that for some $0 < \varepsilon < \frac{q}{2\ell + 1} - 1$ and some sequences $\{k_j\}_{j\geq 0}$, $\{\gamma_j\}_{j\geq 0}$ satisfying  \eqref{growth_condition_kn_1}, \eqref{growth_condition_kn_2}, \eqref{growth_condition_kn_3} and \eqref{definition_gammaj} we have the following
\begin{itemize}
\item[$i)$] Each $\varphi_0^{(N)}$ satisfies
\begin{equation*}
\left\|\varphi_0^{(N)}\right\|_{\dot{\mathcal{F}}_q^{\frac{2\ell-1}{2\ell+1},p}}  \leq  C \left( \sum_{j=N}^{(1+\delta)N} \gamma_j^q k_j^{\frac{2\ell-1}{2\ell+1} q}\right)^{1/q},
\end{equation*}
for some $C = C(\ell, p, q) >0$.
\item[$ii)$] For $1\leq k < \ell$, $f_{k}$ satisfies
\begin{multline*}
\left\|f_{k}^{(N)}(t)\right\|_{\dot{\mathcal{F}}^{\frac{2\ell-1}{2\ell+1},p}_q}
 \leq \frac{C}{t^{2}}  \left(\sum_{j=N}^{(1+\delta)N}  \gamma_j^{(2k+1)q}  k_j^{(2k-2+\frac{2\ell-1}{2\ell+1})q} \right)^{1/q}\\
+\frac{C}{t^{\frac{2\ell-1}{2\ell+1}+2} k_N}  \left(\sum_{j=N}^{(1+\delta)N}\gamma_j  k_j^{\frac{2k-1}{2k+1}}\right)^{2k+1},
\end{multline*}
for some $C = C(k,\ell,p,q) >0$.
\item[$iii)$] For $k = \ell$
\begin{multline*}
\left\|f_{\ell}^{(N)}(t)\right\|_{\dot{\mathcal{F}}^{\frac{2\ell-1}{2\ell+1},p}_q} \geq C \sum\limits_{j=N}^{(1+\delta)N} \gamma_j^{2\ell+1} k_j^{2\ell-1} -  \frac{C}{t^{\frac{2\ell-1}{2\ell+1}+2}  k_N} \sum_{j=N}^{(1+\delta)N} \gamma_j^{2\ell+1}  k_j^{2\ell-1} \\
-  \frac{C}{t^{\frac{2\ell-1}{2\ell+1}+2}  k_N} \left(\sum_{j=N}^{(1+\delta)N} \gamma_j  k_j^{\frac{2\ell-1}{2\ell+1}}\right)^{2\ell+1},
\end{multline*}
for some $C = C(\ell,p,q) >0$.
\end{itemize}
Then, given $R>0$ and $T>0$ there exists $0<\tilde{T} \leq T$ and $\varphi_0 \in \dot{\mathcal{F}}^{\frac{2\ell-1}{2\ell+1},p}_q(\R)$ such that the solution $f$ of \eqref{muskat_problem_truncated_second_iteration} with $\varphi = \varphi_0$ satisfies
\begin{equation*}
\|f(\tilde{T})\|_{\dot{\mathcal{F}}^{\frac{2\ell-1}{2\ell+1},p}_q} > R \text{ and } \|\varphi_0\|_{\dot{\mathcal{F}}^{\frac{2\ell-1}{2\ell+1},p}_q}< 1/R.
\end{equation*}
\end{theorem}

\begin{proof}[Proof of Theorem \ref{thm_inflation_higher_order_muskat}]
First, by definition of the $\dot{\mathcal{F}}^{m,p}_q$ norm \eqref{definition_norm_non_periodic} it is easy to see that
\begin{equation*}
\left\|e^{-t\Lambda} \varphi\right\|_{\dot{\mathcal{F}}^{\frac{2\ell-1}{2\ell+1},p}_q} \leq \| \varphi\|_{\dot{\mathcal{F}}^{\frac{2\ell-1}{2\ell+1},p}_q} .
\end{equation*}
Next, let $\tilde{T} = \min\{T,\bar{T}\}$ and consider $f= \sum_{j=0}^\ell f_{j}$. Then, by using assumptions $ii)$ and $iii)$ we get the estimate
\begin{align*}
\left\|f(\tilde{T})\right\|_{\dot{\mathcal{F}}^{\frac{2\ell-1}{2\ell+1},p}_q} &\geq \|f_{\ell}\|_{\dot{\mathcal{F}}^{\frac{2\ell-1}{2\ell+1},p}_q}  - \sum_{k=1}^{\ell-1}\|f_{k}\|_{\dot{\mathcal{F}}^{\frac{2\ell-1}{2\ell+1},p}_q}  - \|e^{-t \Lambda}\varphi\|_{\dot{\mathcal{F}}^{\frac{2\ell-1}{2\ell+1},p}_q} \\
&\geq 
C_1 \sum\limits_{j=N}^{(1+\delta)N} \gamma_j^{2\ell+1} k_j^{2\ell-1}
-   \frac{C_2}{\tilde{T}^{\frac{2\ell-1}{2\ell+1}+2}  k_N} \sum\limits_{j=N}^{(1+\delta)N} \gamma_j^{2\ell+1}  k_j^{2\ell-1} \\
& \hspace{2cm}-  \frac{C_3}{\tilde{T}^{\frac{2\ell-1}{2\ell+1}+2}  k_N} \left(\sum\limits_{j=N}^{(1+\delta)N} \gamma_j  k_j^{\frac{2\ell-1}{2\ell+1}}\right)^{2\ell+1}\\
&\hspace{2cm}- \sum_{k=1}^{\ell-1} \frac{C_4}{\tilde{T}^{2}}  \left(\sum\limits_{j=N}^{(1+\delta)N}  \gamma_j^{(2k+1)q}  k_j^{(2k-2+\frac{2\ell-1}{2\ell+1})q} \right)^{1/q}\\
&\hspace{2cm} - \sum_{k=1}^{\ell-1} \frac{C_5}{\tilde{T}^{\frac{2\ell-1}{2\ell+1}+2} k_N}  \left(\sum\limits_{j=N}^{(1+\delta)N} \gamma_j  k_j^{\frac{2k-1}{2k+1}}\right)^{2k+1}\\
&\hspace{2cm} - \| \varphi\|_{\dot{\mathcal{F}}^{\frac{2\ell-1}{2\ell+1},p}_q}\\
&=  I_1 - I_2 - I_3-  I_4- I_5 - I_6 .
\end{align*}

We estimate each $I_i$, $i=1 , \cdots, 6$ separately.
\begin{description}
\item[Estimate for $I_1$:]
\begin{equation*}
I_1 = C_1 \sum\limits_{ j=N}^{(1+\delta)N} \gamma_j^{2\ell+1} k_j^{2\ell-1} = C_1  \sum\limits_{ j=N}^{(1+\delta)N} \frac{1}{j^{\frac{(1+\varepsilon)(2\ell+1)}{q}}} \to \infty \text{ as }N \to \infty, 
\end{equation*}
here the sequence diverges because assumption in $\varepsilon$ imply $\frac{(1+\varepsilon)(2\ell+1)}{q} <1$. Therefore by taking $N_1 > \tilde{N}$ large enough we get that 
\begin{equation*}
C_1 \sum\limits_{ j=N}^{(1+\delta)N} \gamma_j^{2\ell+1} k_j^{2\ell-1} > 4R,~ \forall N \geq N_1.
\end{equation*}
\item[Estimate for $I_2$:] 
$$I_2 = \frac{C_2}{\tilde{T}^{\frac{2\ell-1}{2\ell+1}+2}  k_N} \sum\limits_{j=N}^{(1+\delta)N}\gamma_j^{2\ell+1}  k_j^{2\ell-1}.$$
It is easy to see that $|I_2| = \frac{C_2}{k_N C_1\tilde{T}^{2+\frac{2\ell-1}{2\ell+1}} } |I_1|$, then there exists $N_2 \geq N_1$ large enough such that
\begin{equation*}
|I_2| \leq \frac{1}{2} |I_1|,\quad \forall N \geq N_2.
\end{equation*}

\item[Estimate for $I_3$:]
\begin{align*}
I_3  &= \frac{C_3}{\tilde{T}^{\frac{2\ell-1}{2\ell+1}+2}  k_N}  \left( \sum_{j=N}^{(1+\delta)N} \gamma_j  k_j^{\frac{2\ell-1}{2\ell+1}} \right)^{2\ell+1} \\
&= \frac{C_3}{\tilde{T}^{\frac{2\ell-1}{2\ell+1}+2}   k_N}  \left( \sum_{j=N}^{(1+\delta)N} \frac{1}{j^{\frac{2\ell-1}{2\ell+1} \frac{1+\varepsilon}{q}}} \right)^{2\ell+1}.
\end{align*}
Because as $N \to \infty$ the sequence inside the parenthesis diverge, in order to bound $I_3$ we make use that we have a factor of $k_N$ in the denominator. By the assumption \eqref{growth_condition_kn_2} we know that the sequence $k_N$ satisfy
\begin{equation*}
\sum_{j=N}^{(1+\delta)N} \frac{1}{j^{\frac{2\ell-1}{2\ell+1} \frac{1+\varepsilon}{q}}} < \frac{1}{N}  k_N^{\frac{1}{2\ell+1}},
\end{equation*}
which means that we can take $N_3 \geq N_2$ large enough such that
\begin{equation*}
|I_3| \leq \frac{C_3}{\tilde{T}^{2+\frac{2\ell-1}{2\ell+1}} } \frac{1}{N}< 1, ~ \forall N \geq N_3.
\end{equation*}

\item[Estimate for $I_4$:]
Substituting the definition of $\gamma_j$ we get the following identity
\begin{align*}
J_{N,k} &= \sum_{j=N}^{(1+\delta)N}  \gamma_j^{(2k+1)q}  k_j^{(2k-2+\frac{2\ell-1}{2\ell+1})q}\\ 
&= \sum_{j=N}^{(1+\delta)N}  \frac{1}{k_j^{q(2k+1)\frac{2\ell-1}{2\ell+1} - q(2k-2+\frac{2\ell-1}{2\ell+1})}} \frac{1}{j^{(2k+1)(1+\varepsilon)}},
\end{align*}
next, using that $\frac{2\ell-1}{2\ell+1}<1$ we get the following bound for the exponent of $k_j$
\begin{align*}
E &=q(2k+1)\frac{2\ell-1}{2\ell+1} - q\left(2k-2+\frac{2\ell-1}{2\ell+1}\right) \\
&= q(2k+1)\left( \frac{2\ell-1}{2\ell+1}   -  \frac{(2k-2+\frac{2\ell-1}{2\ell+1})}{2k+1}\right)\\
&= q(2k+1)\left( \frac{2\ell-1}{2\ell+1}   -  \frac{2k-1}{2k+1}\right)>0,
\end{align*}
and therefore we conclude that $J_{N,k} \to 0$ as $N\to \infty$, moreover because $\ell$ is fixed we can take $N_4 \geq N_3$ large enough such that
\begin{equation*}
I_4 = \frac{C_4}{\tilde{T}^{2}} \sum_{k=1}^{\ell-1}  J_{N,k}^{1/q} <1, ~ \forall N \geq N_4.
\end{equation*}

\item[Estimate for $I_5$:]
\begin{equation*}
 L_{N,k} = \sum_{j=N}^{(1+\delta)N}  \gamma_j  k_j^{\frac{2k-1}{2k+1}} = \sum_{j=N}^{(1+\delta)N}  \frac{1}{k_j^{\frac{2\ell-1}{2\ell+1} - \frac{2k-1}{2k+1}}} \frac{1}{j^{\frac{1+\varepsilon}{q}}} \to 0 \text{ as } N\to \infty
\end{equation*}
then we can take $N_5 \geq N_4$ large enough such that 
\begin{equation*}
I_5 = \sum_{k=1}^{\ell-1} \frac{C_5}{\tilde{T}^{2+\frac{2\ell-1}{2\ell+1}} k_N}  L_{N,k}^{2k+1} < 1, ~ \forall N \geq N_5.
\end{equation*}
\item[Estimate for $I_6$:] From hypothesis i) we have the estimate
\begin{equation*}
I_6 =  \| \varphi\|_{\dot{\mathcal{F}}^{\frac{2\ell-1}{2\ell+1},p}_q} \leq C_6 \left( \sum_{j=N}^{(1+\delta)N} \gamma_j^q k_j^{\frac{2\ell-1}{2\ell+1} q}\right)^{1/q},
\end{equation*}
substituting the definition of $\gamma_j$  we know that
\begin{equation*}
\sum_{j=N}^{(1+\delta)N} \gamma_j^q k_j^{\frac{2\ell-1}{2\ell+1} q} = \sum_{j=N}^{(1+\delta)N} \frac{1}{j^{1+\varepsilon}} \to 0 \text{ as } N\to \infty,
\end{equation*}
therefore we can take $N_6 \geq N_5$ large enough such that 
\begin{equation*}
I_6=  \| \varphi\|_{\dot{\mathcal{F}}^{\frac{2\ell-1}{2\ell+1},p}_q}  \leq C_6 \left( \sum_{j=N}^{(1+\delta)N} \gamma_j^q k_j^{\frac{2\ell-1}{2\ell+1} q}\right)^{1/q} <\frac{1}{R}, \forall N \geq N_6.
\end{equation*}
\end{description}

Finally by combining the previous estimates we obtain
\begin{equation*}
\|f(\tilde{T})\|_{\dot{\mathcal{F}}^{\frac{2\ell-1}{2\ell+1},p}_q}\geq 2R - 4 \geq R,
\end{equation*}
for all $N\geq N_6$. This completes the proof of Theorem
\ref{thm_inflation_higher_order_muskat}.
\end{proof}

\section{Verification of Hypothesis Theorem \ref{thm_inflation_higher_order_muskat}}\label{section_non_periodic}

The goal of this section is to prove Theorem \ref{thm_norm_inflation_finite_expansion_muskat} by verifying the hypothesis of Theorem \ref{thm_inflation_higher_order_muskat}. This is done in three steps: first, hypothesis (i) is verified in Lemma \ref{lemma_size_initial_data_higher_order_muskat} by estimating the size of the initial data given by \eqref{intial_condition_higher_order_muskat}. Next, using the same initial data as before, we verify hypothesis (ii) in Lemma 
\ref{lemma_estimate_lower_order_terms_higher_order_muskat} by using the decomposition \eqref{muskat_problem_truncated_second_iteration_decomposition} to obtain that the terms $f_k$ for $k< \ell$ are small and finally hypothesis (iii) is verified in Lemma \ref{lemma_estimate_main_term_higher_order_muskat} to  obtain that the term $f_\ell$ is dominant.

A definition that will be useful in the proofs is the following.
\begin{definition}
Consider $E: C([0,T],C_0^\infty(\R)) \to C([0,T],C_0^\infty(\R))$ defined by
\begin{equation*}
E(g)(x, t) := \int_0^t e^{-(t-\tau) \Lambda} \mathcal{F}^{-1}(g)(x,\tau) d\tau.
\end{equation*}
Via approximations this definition can be extended continuously to all $g \in C([0,T],S')$ to obtain $E: C([0,T],S') \to C([0,T],S')$.
\end{definition}

\begin{lemma}[Size of the Initial Data]\label{lemma_size_initial_data_higher_order_muskat} Let $m\geq 0$, $p\geq 1$ and $q\geq 1$, and consider $\varphi^{(N)}$ given by \eqref{intial_condition_higher_order_muskat}, then 
\begin{equation*}
\left\|\varphi^{(N)}\right\|_{\dot{\mathcal{F}}_q^{m,p}} \leq    C \left( \sum_{j=N}^{(1+\delta)N} \gamma_j^q k_j^{m q}\right)^{1/q},
\end{equation*}
for a constant $C = C(m,q) >0$.
\end{lemma}
\begin{proof}[Proof of Lemma \ref{lemma_size_initial_data_higher_order_muskat}] From the definition of the $\dot{\mathcal{F}}^{m,p}_q$ norm in \eqref{definition_norm_non_periodic} we have
\begin{multline*}
\left\|\varphi^{(N)}\right\|_{\dot{\mathcal{F}}_q^{m,p}}  \leq C  \Bigg[ \sum_{j=N}^{(1+\delta)N} \gamma_j^q \Bigg( \left( \int_\R | \xi |^{m p}|P_{k_j}|^p d\xi \right)^{q/p} \\
+ \left(\int_\R | \xi |^{m p}|P_{2 \ell k_j + M}|^p d\xi\right)^{q/p}\Bigg) \Bigg]^{1/q},
\end{multline*}
here we used that because of \eqref{growth_condition_kn_1} on each annulus $C_k$ at most one among $\{P_{k_j}, P_{2\ell k_j+M}\}_{j=N}^{(1+\delta)N}$ is nonzero. Next using that the support of $P_A(\xi)$
for $A>2$ can only intersect at most two annulus $C_k$ we can bound
\begin{align*}
\left\|\varphi^{(N)}\right\|_{\mathcal{\dot{F}}_q^{m,p}}& \leq  C \left( \sum_{j=N}^{(1+\delta)N} \gamma_j^q \Big( (k_j+1)^{mq} +  (2 \ell k_j + M+1)^{m q} \Big) \right)^{1/q}\\
& \leq   C \left( \sum_{j=N}^{(1+\delta)N} \gamma_j^q k_j^{mq}\right)^{1/q},
\end{align*}
for some constant $C = C(m,\ell) > 0$.
\end{proof}
\begin{lemma}[Estimate for $1 \leq k<\ell$]\label{lemma_estimate_lower_order_terms_higher_order_muskat} Let $f_{k }$ for $1 \leq k <\ell$ as defined by \eqref{muskat_problem_truncated_second_iteration_decomposition} with $\varphi =\varphi^{(N)}$ given by \eqref{intial_condition_higher_order_muskat}, and  $t>0$ such that $t M < 1$ and $t k_N > 1$ then for $0 < m < 1$ we have the estimate
\begin{multline*}
\|f_{k}\|_{\dot{\mathcal{F}}^{m,p}_q}
\leq \frac{C}{t^{2}}  \left(\sum_{j=N}^{(1+\delta)N}  \gamma_j^{(2k+1)q}  k_j^{(2k-2+m)q} \right)^{1/q}\\+\frac{C}{t^{m+2} k_N }  \left(\sum_{j=N}^{(1+\delta)N}  \gamma_j  k_j^{\frac{2k-1}{2k+1}}\right)^{2k+1},
\end{multline*}
for $1\leq k < \ell$ and a constant $C = C(m,p,q,k,\ell)>0$.
\end{lemma}

\begin{lemma}[Estimate for $k = \ell$]\label{lemma_estimate_main_term_higher_order_muskat} 
Under the same assumptions as Lemma \ref{lemma_estimate_lower_order_terms_higher_order_muskat} for $0<m<1$ we have the estimate
\begin{multline*}
\|f_{\ell}\|_{\dot{\mathcal{F}}^{m,p}_q} \geq C \sum\limits_{j=N}^{(1+\delta)N} \gamma_j^{2\ell+1} k_j^{2\ell-1} -  \frac{C}{t^{m+2}  k_N} \sum_{j=N}^{(1+\delta)N} \gamma_j^{2\ell+1}  k_j^{2\ell-1} \\
-  \frac{C}{t^{m+2}  k_N} \left(\sum_{j=N}^{(1+\delta)N} \gamma_j  k_j^{\frac{2\ell-1}{2\ell+1}}\right)^{2\ell+1},
\end{multline*}
for a constant $C = C(m,p,q,\ell)>0$.
\end{lemma}	
\begin{remark}
This Lemma is one of the more delicate steps in the proof since it requires a lower bound for a highly oscillatory singular integral.
\end{remark}
For the proof of Lemma \ref{lemma_estimate_lower_order_terms_higher_order_muskat} an important technical tool is to be able to estimate the fo\-llo\-wing integral which will appear several times in our computations.
\begin{definition}\label{definition_gamma_multidimensional}
The function $\Gamma_{2k+1}: \R^{2k+1} \to \C$ is defined by the following integral
\begin{equation*}
\Gamma_{2k+1}(A_1,\cdots , A_{2k+1}) := i\int_{\R} m_\alpha(A_1)m_\alpha(A_2)\cdots m_\alpha(A_{2k})m_\alpha(A_{2k+1})d\alpha, 
\end{equation*}
where $m_\alpha(A) = \frac{1-e^{-2\pi i\alpha A}}{\alpha}$. 
\end{definition}
Some important properties of $\Gamma_{2k+1}$ are given by the following Lemma.

\begin{lemma}[Properties of $\Gamma_{2k+1}$]\label{properties_gamma_higher_dimension}
Let $k\geq 1$, then the function $\Gamma_{2k+1}$ given by Definition \ref{definition_gamma_multidimensional} satisfy the following
\begin{itemize}
\item[i)] $\Gamma_{2k+1}(A_1, \cdots, A_{2k+1})$ is given explicitly by
\begin{multline*}
\Gamma_{2k+1}(A_1, \cdots, A_{2k+1}) =  \frac{(2\pi)^{2k} \pi}{(2k)!}\Big( A_1^{2k-1} |A_1| + \cdots + A_{2k+1}^{2k-1}|A_{2k+1}|\\
- (A_1+A_2)^{2k-1} |A_1+A_2| - (A_1+A_3)^{2k-1} |A_1+A_3|\\
+ \cdots + (A_1 + \cdots + A_{2k+1})^{2k-1}|A_1 + \cdots + A_{2k+1}| \Big),
\end{multline*}
\item[ii)] $\Gamma_{2k+1}( c A_1, \cdots , c A_{2k+1}) = c^{2k} \sgn(c) \Gamma_{2k+1}( A_1, \cdots, A_{2k+1})$ for $c\in \R$,
\item[iii)] $\Gamma_{2k+1}(A_1, \cdots ,A_{2k+1})=0$ if $A_i >0$ for all $i = 1, \cdots, 2k+1$,
\item[iv)] $\Gamma_{2k+1}(A_{\sigma(1)},\cdots, A_{\sigma(2k+1)}) = \Gamma_{2k+1}(A_1,\cdots, A_{2k+1})$ where $\sigma$ is any permutation the set $\{1,2 \cdots, 2k+1\}$,
\item[v)] $|\Gamma_{2k+1}( A_1, \cdots , A_{2k+1}) |\leq (2 \pi)^{2k+1} |A_1| |A_2| \cdots |A_{2k}|$ . Notice that there are only $2k$ terms in the right hand side and not $2k+1$.
\item[vi)] $|\Gamma_{2k+1}(A_1, \cdots , A_{2k+1}) | \leq (2\pi)^{2k+1}|A_1 \cdots A_{2k+1}|^{\frac{2k}{2k+1}}$,\\
$|\Gamma_{2k+1}(A_1, \cdots , A_{2k+1}) | \leq (2 \pi)^{2k+1} \min\limits_{1\leq j\leq 2k+1} |A_1  \cdots A_{j-1}  A_{j+1} \cdots A_{2k+1}|$.
\item[vii)] Let $x_i \in [A_i -1,A_i +1]$, $|A_i| > 2$, $i=1 , \cdots , 2k+1$ then
\begin{multline*}
\left|\Gamma_{2k+1}(x_1,\cdots, x_{2k+1}) - \Gamma_{2k+1}(A_1, \cdots ,A_{2k+1})\right| \\
\leq  C(|A_1|^{2k-1}+ \cdots + |A_{2k+1}|^{2k-1}),
\end{multline*}
for a constant $C= C(k) >0$.
\end{itemize}
\end{lemma}

\begin{proof}[Proof of Lemma \ref{properties_gamma_higher_dimension}]
$i)$ is obtained by integration by parts. $ii)$ is direct consequence of the explicit formula in part $i)$. $iii)$ is obtained from Definition \ref{definition_gamma_multidimensional} using that 
\begin{equation*}
m_\alpha(A) = \frac{1-e^{-2\pi i\alpha A}}{\alpha} = \frac{2 i e^{-\pi i\alpha A}}{\alpha} \sin(\pi \alpha A),
\end{equation*}
substituting this on Definition \ref{definition_gamma_multidimensional} we get
\begin{multline*}
\Gamma_{2k+1}(A_1,\cdots , A_{2k+1}) =i (2i)^{2k+1}\int_\R e^{-\pi i \alpha(A_1 + \cdots + A_{2k+1})} \\
\times \frac{\sin(\pi \alpha A_1)}{\alpha} \cdots \frac{\sin(\pi \alpha A_{2k+1})}{\alpha}d \alpha.
\end{multline*}
To see that this integral is zero, note that for $A_i > 0$
\begin{equation*}
\mathcal{F}^{-1}\left(\chi_{[-A_i/2,A_i/2]}\right) = \frac{\sin(\pi \alpha A_i)}{\pi\alpha},
\end{equation*}
this means that the integral can be seen as a Fourier transform of a product at the point  $\xi = -\frac{A_1+\cdots+A_{2k+1}}{2}$
\begin{align*}
\Gamma_{2k+1}(A_1,\cdots , A_{2k+1}) &= (2\pi)^{2k+1} (-1)^{k+1} \chi_{[-A_1/2,A_1/2]} * \chi_{[-A_2/2,A_2/2]} * \cdots \\
& \hspace{1cm}* \chi_{[-A_{2k+1}/2,A_{2k+1}/2]}\left(-\frac{A_1+\cdots+A_{2k+1}}{2}\right),
\end{align*}
next, the inclusion $\text{supp} (f* g)\subset \text{supp}(f) + \text{supp}(g)$ imply
\begin{multline*}
\supp\left( \chi_{[-A_1/2,A_1/2]} * \chi_{[-A_2/2,A_2/2]} * \cdots 
* \chi_{[-A_{2k+1}/2,A_{2k+1}/2]} \right)\\
\subset \left[-(|A_1| + \cdots +  |A_{2k+1}|)/2,(|A_1| + \cdots  + |A_{2k+1}|)/2\right].
\end{multline*}
Finally because the convolution of characteristic functions is continuous, we conclude that the convolution is exactly zero $\xi  = -\frac{A_1+\cdots+A_{2k+1}}{2}$ and therefore
\begin{equation*}
\Gamma_{2k+1}(A_1,\cdots , A_{2k+1}) = 0.
\end{equation*}
Part $iv)$ follows directly from the definition. To prove part $v)$ we need to use the integral formula in Definition \ref{definition_gamma_multidimensional} and observe that
\begin{equation*}
\frac{1-e^{-2\pi i\alpha x}}{\alpha} = 2\pi i x \int_0^1 e^{-2\pi i x \alpha(1-t)} dt,
\end{equation*}
applying this to definition \ref{definition_gamma_multidimensional} we get
\begin{align*}
B &= \Gamma_{2k+1}(A_1,A_2, \cdots ,A_{2k+1}) \\
&= i(2\pi i)^{2k} A_1 A_2 \cdots A_{2k} \\
&  \hspace{5mm}\times \int_\R \int_0^1 \cdots \int_0^1\frac{1}{\alpha}  \left(1-e^{-2\pi i\alpha A_{2k+1}}\right)e^{-2\pi i \alpha A_1 (1-t_1)}\\
&  \hspace{5mm}\times e^{-2\pi i \alpha A_2 (1-t_2)}\cdots e^{-2\pi i \alpha A_{2k} (1-t_{2k})} d t_1 \cdots dt_{2k} d\alpha\\
&= i(2\pi i)^{2k} A_1 A_2 \cdots A_{2k} \\
& \times \Big( \int_0^1 \cdots \int_0^1\int_\R \frac{1}{\alpha}\left( e^{-2\pi i \alpha (A_1 (1-t_1) + \cdots + A_{2k}(1-t_{2k}))} -1\right) d\alpha dt_1 \cdots dt_{2k}  \\
&- \int_0^1 \cdots \int_0^1\int_\R \frac{1}{\alpha}\left(1- e^{-2\pi i \alpha (A_1 (1-t_1) + \cdots + A_{2k}(1-t_{2k}) + A_{2k+1}) }\right) d\alpha dt_1 \cdots dt_{2k}\Big),
\end{align*}
next using that $\int_\R \frac{1-e^{-2\pi i \alpha x}}{\alpha} d\alpha = i \pi \sgn(x)$ we get that 
\begin{align*}
|B| &\leq \pi  (2\pi)^{2k} |A_1 \cdots A_{2k}| \int_0^1 \cdots \int_0^1 \Big|-\sgn \left(A_1 (1-t_1) +  \cdots + A_{2k}(1-t_{2k})\right) \\
&\hspace{1cm}+\sgn\left(A_1 (1-t_1)+ \cdots + A_{2k}(1-t_{2k}) + A_{2k+1}\right)\Big| dt_1 \cdots dt_{2k}\\
&\leq  (2\pi)^{2k+1} |A_1 \cdots A_{2k}| .
\end{align*}
Part $vi)$ is obtained from $v)$ and the observation that because of $iv)$ the variable that we omit in the estimate $v)$ can be any variable, and therefore taking the geometric average of the inequalities give us the result. To prove $vii)$ we use that because $\Gamma_{2k+1}(x_1,\cdots,x_{2k+1})$ is differentiable, we can use the mean value theorem to get the estimate. From $i)$ we can bound the partial derivative using
\begin{align*}
\frac{\partial}{\partial x_i}\Gamma_{2k+1}(x_1,\cdots,x_{2k+1}) &\leq 2k (2^{2k+2}-1)(|x_1| + \cdots +|x_{2k+1}|)^{2k-1}\\
&\leq C (|x_1|^{2k-1} + \cdots +|x_{2k+1}|^{2k-1}),
\end{align*}
for a constant $C = C(k) >0$. Next, by the mean value theorem
\begin{align*}
J &= \left|\Gamma_{2k+1}(x_1,\cdots, x_{2k+1}) - \Gamma_{2k+1}(A_1, \cdots ,A_{2k+1})\right|\\
&\leq \sum_{i=1}^{2k+1} \left|\frac{\partial}{\partial x_i}\Gamma_{2k+1}(y_1,\cdots, y_{2k+1})\right|\\
&\leq (2k+1) C (|A_1 + 1|^{2k-1} + \cdots + |A_{2k+1} + 1|^{2k-1})\\
&\leq C_2 (|A_1|^{2k-1} + \cdots + |A_{2k+1}|^{2k-1}),
\end{align*}
for a constant $C_2 = C_2(k) >0$. This concludes the proof of Lemma  \ref{properties_gamma_higher_dimension}
\end{proof}
\noindent

Now we proceed to the proof of Lemma \ref{lemma_estimate_lower_order_terms_higher_order_muskat}.

\begin{proof}[Proof of Lemma \ref{lemma_estimate_lower_order_terms_higher_order_muskat}]
First, taking Fourier transform of  equation \eqref{terms_tk_expasion_taylor_muskat_real_line} and subs\-ti\-tu\-ting that in \eqref{equation_toy_system_muskat_higher_order} we get for $1\leq k < \ell$ 
\begin{equation*}
\hat{f}_{k}(\xi,t) = \frac{(-1)^{k}}{\pi(2k+1)}E\left((2\pi i \xi)\int_\R (m_\alpha \hat{f}_0)^{*(2k+1)} d\alpha\right),
\end{equation*}
next substituting $\varphi = \varphi^{(N)}$ given by \eqref{intial_condition_higher_order_muskat} in \eqref{muskat_problem_truncated_second_iteration_decomposition} we obtain the following formula for $f_0$
\begin{equation*}
\hat{f}_0 = \sum_{j=N}^{(1+\delta)N} \gamma_j e^{-2\pi t |\xi|} (P_{k_j}+P_{2\ell k_j + M}),
\end{equation*}
substituting in $f_k$ we get
\begin{align}
\hat{f}_{k}(\xi,t) &= \frac{(-1)^{k}}{\pi(2k+1)} E\left(\sum\limits_{\substack{N\leq s_\eta \leq (1+\delta)N\\ 1\leq \eta \leq 2k+1}} \sum\limits_{c_\eta^{s_\eta}\in \Lambda(k_{s_\eta})}\gamma_{s_1} \cdots \gamma_{s_{2k+1}} R_{c_1^{s_1}\cdots c_{2k+1}^{s_{2k+1}}}(\xi,t)\right),\label{full_expansion_f2k1}\\
&= \frac{(-1)^{k}}{\pi(2k+1)}\left(E(J_{k}) + E(HF)\right) \label{eqn_Lemma2_5_I2k1},
\end{align}
where
\begin{equation}\label{definition_set_lks}
\Lambda(k_s) = \{\pm k_s, \pm (2\ell k_s + M)\}.
\end{equation}
and
\begin{equation}\label{definition_J1_diagonal_terms}
J_{k}(\xi,t) = \sum\limits_{j =N}^{(1+\delta)N} \sum\limits_{\substack{c_i^j\in \Lambda(k_j)\\ i = 1,\cdots , 2k+1}} \gamma_j^{2k+1} R_{c_1^j\cdots c_{2k+1}^j}(\xi,t) ,
\end{equation}
\begin{equation}\label{definition_HF_higher_order_muskat}
HF(\xi,t)=  \sum\limits_{\substack{N\leq s_\eta \leq (1+\delta)N\\ \text{not all equal }\\ 1\leq \eta \leq 2k+1}} \sum\limits_{c_\eta^{s_\eta}\in \Lambda(k_{s_\eta})}\gamma_{s_1} \cdots \gamma_{s_{2k+1}} R_{c_1^{s_1}\cdots c_{2k+1}^{s_{2k+1}}}(\xi,t),
\end{equation}
\begin{equation}\label{definition_R}
R_{c_1^{s_1}\cdots c_{2k+1}^{s_{2k+1}}}(\xi,t) = (2\pi i \xi)\int_\R (e^{-2\pi t|\cdot|}m_\alpha \chi_{c_1^{s_1}})* \cdots *(e^{-2\pi t|\cdot|}m_\alpha \chi_{c_{2k+1}^{s_{2k+1}}})d\alpha.
\end{equation}
Here $HF$ represent the off-diagonal terms in the sum in  \eqref{full_expansion_f2k1}, which we expect to have high frequency and therefore its norm should decay faster than the lower frequency terms, this should make this term easier to bound. A formula for $R_{c_1^{s_1}\cdots c_{2k+1}^{s_{2k+1}}}(\xi,t)$ that will be useful later is obtained by expanding the convolution and writing it in terms of the function $\Gamma_{2k+1}$ given by definition \ref{definition_gamma_multidimensional}
\begin{align}
R_{c_1^j\cdots c_{2k+1}^j}(\xi,t) &=(2\pi \xi)\int d\xi_1 \cdots \int d \xi_{2k} \Gamma_{2k+1}(\xi-\xi_1, \xi_1-\xi_2, \cdots,\xi_{2k})\nonumber \\
&\hspace{1cm}\times e^{-2\pi t|\xi-\xi_1|}\chi_{c_1^j}(\xi - \xi_1) e^{-2\pi t|\xi_1-\xi_2|}\chi_{c_2^j}(\xi_1 - \xi_2) \nonumber \\
&\hspace{1cm}\times \cdots e^{-2\pi t|\xi_{2k}|}\chi_{c_{2k+1}^j}(\xi_{2k}). \label{expansion_inner_term_higher_order_muskat}
\end{align}
Our main estimate is based in equation \eqref{eqn_Lemma2_5_I2k1} by obtaining estimates in the $\dot{\mathcal{F}}^{m,p}_{q}$ norm for $J_k$ and $HF$ given by the next two Lemmas.

\begin{lemma}[Estimate diagonal terms in Lemma \ref{lemma_estimate_lower_order_terms_higher_order_muskat}]\label{estimate_higher_order_first_part} Let $1\leq k < \ell$, $0< t \leq 1$ and consider $J_{k}$ as defined by \eqref{definition_J1_diagonal_terms}, then
\begin{equation*}
\left\|  E(J_{k})\right\|_{\dot{\mathcal{F}}^{m,p}_{q}} 
\leq \frac{C}{t^{2}} \left(\sum_{j=N}^{(1+\delta)N}  \gamma_j^{(2k+1)q}  k_j^{(2k-2+m)q}  \right)^{1/q},
\end{equation*}
for some constant $C= C(k,\ell,M,p,q)>0$.
\end{lemma}

\begin{lemma}[Estimate off diagonal terms in Lemma  \ref{lemma_estimate_lower_order_terms_higher_order_muskat}]\label{lemma_higher_frequency_muskat_higher_order}
Let $HF$ as defined by \eqref{definition_HF_higher_order_muskat} and $0<t\leq 1$, then
\begin{equation*}
\|E(HF)\|_{\dot{\mathcal{F}}^{m,p}_q}\leq\frac{C}{t^{m+2 }k_N}\left(\sum_{j=N}^{(1+\delta)N} \gamma_j k_j^{\frac{2k-1}{2k+1}}\right)^{2k+1},
\end{equation*}
for some constant $C= C(k,\ell,M,p,q)>0$.
\end{lemma}

\begin{proof}[Proof of Lemma \ref{estimate_higher_order_first_part}]
From equation  \eqref{expansion_inner_term_higher_order_muskat} we get that in order to estimate $R_{c_1^j \cdots c_{2k+1}^j}$ we need to estimate $\Gamma_{2k+1}(\xi-\xi_1, \xi_1-\xi_2, \cdots,\xi_{2k})$ in the region where \begin{equation*}
\chi_{c_1^j}(\xi - \xi_1) \chi_{c_2^j}(\xi_1 - \xi_2) \cdots \chi_{c_{2k+1}^j}(\xi_{2k})\neq 0.
\end{equation*}
Here we have two cases: first, when all $c_i^j$,  $i=1, \cdots, 2k+1$ have the same sign, then in the region we are interested we have that all $\xi - \xi_1$, $\xi_1 -\xi_2$, $\cdots$, $\xi_{2k}$ have the same sign and therefore we can apply Lemma \ref{properties_gamma_higher_dimension} part iii) to obtain that  $R_{c_1^j \cdots c_{2k+1}^j} = 0$.

On the other hand, when not all $c_i^j$, $i=1, \cdots, 2k+1$ have the same sign the estimate is more delicate. Because $|c_i^j|\leq 2\ell k_j +M$ for $c_i^j \in \Lambda(k_j)$, we can use Lemma \ref{properties_gamma_higher_dimension} parts v) and vii) to get the estimate 
\begin{align}
|\Gamma_{2k+1}(\xi-\xi_1, \xi_1-\xi_2, \cdots,\xi_{2k})| &\leq  |\Gamma_{2k+1}(c_1^j,\cdots,c_{2k+1}^j)| \nonumber\\
& \hspace{1cm} + O(|c_1^j|^{2k-1}+\cdots + |c_{2k+1}^j|^{2k-1}) \nonumber\\
&\leq  C (2\ell k_j + M)^{2k} + O(|k_j|^{2k-1}) \nonumber \\	
&\leq  C k_j^{2k} , \label{easy_upper_bound_Gamma}
\end{align}
for a constant $C = C(k,M,\ell) >0$. By applying \eqref{easy_upper_bound_Gamma} to \eqref{expansion_inner_term_higher_order_muskat} we get
\begin{align}
|R_{c_1^j\cdots c_{2k+1}^j}| &\leq C (2\pi |\xi|)k_j^{2k} \int_\R \cdots \int_\R (e^{-2\pi t|\xi - \xi_1|}\chi_{c_1^j}(\xi-\xi_1))(e^{-2\pi t|\xi_1 - \xi_2|}\nonumber \\
 & \hspace{1cm}\times\chi_{c_2^j}(\xi_1-\xi_2)) \cdots (e^{-2\pi t|\xi_{2k}|}\chi_{c_{2k+1}^j}(\xi_{2k})) d\xi_1 \cdots d\xi_{2k} \nonumber \\
&\leq C |\xi|
k_j^{2k} e^{-2\pi t (|c_1^j|+\cdots+|c_{2k+1}^j|-(2k+1) )}  h_{c_1^j \cdots c_{2k+1}^j}(\xi), \label{estimate_R_lemma_lower_order_terms_non_periodic}
\end{align}
where
\begin{equation}\label{definition_h_higher_order_muskat}
h_{c_1^j \cdots c_{2k+1}^j}(\xi)=\left(\chi_{c_1^j}*\chi_{c_2^j}* \cdots * \chi_{c_{2k+1}^j}\right)(\xi).
\end{equation}

To continue with the estimate we need two Lemmas. The first one provides a precise notion on how a convolution of characteristic functions can be compared with a single characteristic function and will be used to estimate the term  $h_{c_1^j \cdots c_{2k+1}^j}(\xi)$. The second one give us an estimate for the size of the sum $c_1+ \cdots +c_{2k+1}$.

\begin{lemma}[Convolutions of characteristic functions]\label{estimate_iterated_convolution_characteristic_functions}
Let $c_1, \cdots, c_k \in \R$ and $\chi_A = 1_{[A-1,A+1]}$,  then
\begin{itemize}
\item[(i)]
$
\chi\left(\xi-(c_1 + \cdots +c_{k})\right)\leq \chi_{c_1}* \chi_{c_2}* \cdots * \chi_{c_{k}}(\xi) \leq 2^k \chi\left(\frac{\xi-(c_1 + \cdots +c_{k})}{k}\right),
$
\item[(ii)] If $|c_1+ \cdots +c_k| > k+1$, $p\geq 1$, $q \geq 1$, $m\geq 0$ then
\begin{equation*}
\|\chi_{c_1}* \chi_{c_2}* \cdots * \chi_{c_{k}}\|_{\dot{\mathcal{F}}^{m,p}_q}\leq C |c_1 + \cdots + c_{k} + k|^{m},
\end{equation*}
for a constant $C= C(k,p) > 0$.
\end{itemize}
\end{lemma}
\begin{proof}[Proof of Lemma \ref{estimate_iterated_convolution_characteristic_functions}]
For the lower bound in part (i) the key fact is the following inequality
\begin{equation*}
(\chi_A * \chi_B)(\xi) = max\{2-(\xi-A-B),0\} \geq \chi_{A+B}(\xi),
\end{equation*}
and by iterating this inequality we obtain the lower bound. For the upper bound we need two observations: first, for $A, B \subset \R$ we have the inclusion
\begin{equation*}
\supp \left(1_A * 1_B\right) \subset A+B = \{a+b : a\in A, b\in B\},
\end{equation*}
second, the convolution of characteristic functions can be bounded by
\begin{align*}
\chi\left(\frac{\cdot-A}{a}\right) * \chi\left(\frac{\cdot-B}{b}\right) &= \int_\R \chi\left(\frac{\xi - y -A}{a}\right) \chi\left(\frac{y-B}{b}\right) dy\\ &\leq \int_\R \chi\left(\frac{y-B}{b}\right) dy\\
 &= 2b,
\end{align*}
and by symmetry $\chi(\frac{\cdot-A}{a}) * \chi(\frac{\cdot-B}{b}) \leq 2\min\{a,b\}$, iterating this result we obtain the upper bound
\begin{equation*}
\chi_{c_1}* \chi_{c_2}* \cdots * \chi_{c_{k}} \leq 2^k \chi\left(\frac{\xi-(c_1 + \cdots +c_{k})}{k}\right),
\end{equation*}
this completes the proof of part (i). For part (ii) we use the upper bound obtained in part (i)
\begin{align*}
H &= \|\chi_{c_1}* \chi_{c_2}* \cdots * \chi_{c_{k}}\|_{\dot{\mathcal{F}}^{m,p}_{q}} \\
&= \left(\sum_{n\in \Z} \left(\int_{C_n} |\xi|^{mp} |\chi_{c_1}* \chi_{c_2}* \cdots * \chi_{c_{k}}|^p d\xi\right)^{q/p}\right)^{1/q}\\
&\leq \Bigg(\sum_{n\in \Z} \Big(\int_{C_n} |c_1 + \cdots + c_{k} + k|^{mp}\\
& \hspace{1cm}\times \left|2^{k} \chi\left(\frac{\xi-(c_1+ \cdots +c_{k})}{k}\Big)\right|^p d\xi\right)^{q/p}\Bigg)^{1/q}\\
&\leq  2^{k} |c_1 + \cdots + c_{k} + k|^{m} \\
& \hspace{1cm}\times \left( R \left(\int_{\R}  \left| \chi\left(\frac{\xi-(c_1+ \cdots +c_{k})}{k}\right) \right|^p d\xi\right)^{q/p}\right)^{1/q}\\
&\leq  2^{k} R^{1/q} (2k)^{1/p} |c_1 + \cdots + c_{k} + k|^{m} .
\end{align*}
where $R$ is the number of dyadic intervals that intersect the interval
\[
[c_1 + \cdots + c_k - k , c_1+ \cdots + c_k + k],
\]
here we use our assumption that $|c_1 + \cdots + c_k |\geq k +1$ to conclude that $R \leq \log_2 k + 1$, which give us part (ii) and conclude the proof of Lemma \ref{estimate_iterated_convolution_characteristic_functions}.
\end{proof}

\begin{lemma}\label{lemma_higher_order_sums_muskat}
Fix $j\in\N$ and let $c_i \in \Lambda(k_j) = \{\pm k_j, \pm(2\ell k_j + M)\}$, $i=1, \cdots, 2k+1$, with $\{k_n\}_{n=0}^{\infty}$ satisfying \eqref{growth_condition_kn_1} and \eqref{growth_condition_kn_3}, and suppose that not all $c_i$ have the same sign, then
\begin{itemize}
\item[i)] $c_1 + \cdots + c_{2k+1} = \pm M$ if and only if $k=\ell$, exactly one of them is equal to $\pm (2\ell k_j+M)$ and all the others equal to $\mp k_j$.
\item[ii)] $|c_1 + \cdots +c_{2k+1}| \geq k_j/2$,
\item[iii)] $|c_1|+\cdots+|c_{2k+1}|-|c_1 + \cdots +c_{2k+1}| \geq 2k_j$,

\end{itemize}
\end{lemma}

\begin{proof}[Proof of Lemma \ref{lemma_higher_order_sums_muskat}]
For part i) because of condition  \eqref{growth_condition_kn_3} the assumption $c_1+ \cdots + c_{2k+1} = \pm M$ tell us that we need an odd number of terms of the form $\pm (2\ell k_j+M)$, call this number $d$. When we add up all the terms of the form $\pm (2\ell k_j+M)$ the sum must be of the form  $2 n \ell k_j \pm M$, $n\in \Z$ and $n\neq 0$ because we have an odd number of such terms. The remaining terms are of the form $\pm k_j$ and its sum can be written as $m k_j$ where $|m|\leq  2k + 1 -d$ and equal to $2k+1 -d$ if and only if all the terms of the form $\pm k_j$ have the same sign. Then we have the following bound in the size of  $c_1+ \cdots + c_{2k+1}$ 
\begin{multline*} 
|c_1+ \cdots + c_{2k+1}| = |2 n \ell k_j \pm M +  m k_j| \geq (2 |n| \ell - m)k_j -M\\
\geq (2 \ell - (2k + 1 - d) )k_j -M = ((2\ell - 2k) + d-1)k_j -M.
\end{multline*}
We get that the sum of all the terms can only be equal to $M$ if both $(2\ell - 2k)+ (d-1)=0$ and $(2 |n| \ell - m)=0$. From the first condition, because both terms are non negative we get that $k = \ell$ and $d=1$, which also imply that $|n| =1$. From the second condition we get $m = 2\ell$, but that is only possible if all the terms of the form $\pm k_j$ have the same sign. It remains to check  only 4 cases for the terms  $c_1, \cdots, c_{2\ell+1}$
\begin{itemize}
\item[(a)] one equal to $(2\ell k_j + M)$ and all other equal to $k_j$,
\item[(b)] one equal to $(2\ell k_j + M)$ and all other equal to $-k_j$,
\item[(c)] one equal to $-(2\ell k_j + M)$ and all other equal to $k_j$,
\item[(d)] one equal to $-(2\ell k_j + M)$ and all other equal to $-k_j$.
\end{itemize}
Only cases (b) and (c) satisfy that the sum is equal to $\pm M$, which concludes the proof of part i). To prove part ii) we first need that it is impossible to write zero as the sum of $2k+1$ terms using only $\{\pm 1, \pm 2\ell\}$ for $k <\ell$, this can be done in a similar way to part i). Next, we write $c_i = a_i k_j + \epsilon_i$, where $a_i\in \{\pm 1, \pm 2\ell \}$ and $\epsilon_i \in \{-M,0,M\}$. Then
\[
c_1 + \cdots + c_{2k+1} = 
(a_1 + \dots +a_{2k+1}) k_j + (\epsilon_1 + \cdots + \epsilon_{2k+1}).
\]
By the previous observation we see that $|a_1 + \cdots + a_{2k+1}| \geq 1$ and therefore 
\begin{align*}
|c_1 + \cdots + c_{2k+1}| &\geq k_j - |\epsilon_1 + \cdots + \epsilon_{2k+1}|\\
&\geq k_j - (2k+1)M \\
&\geq k_j/2,
\end{align*}
here we used assumption \eqref{growth_condition_kn_3} on the sequence $k_j$. This completes the proof of part ii). Part iii) is consequence of the following proposition.
\begin{proposition}\label{lemma_triangle_inequality_special}
Let $n \geq 2$, $x_1, \cdots,x_n \in \R\setminus \{0\}$ and suppose that not all of them have the same sign, then
\[
|x_1 + \cdots + x_n| \leq |x_1|+ \cdots +|x_n| - 2|x_j|,
\]
for some $j\in \{1, \cdots, n\}$.
\end{proposition}
\begin{proof}[Proof of Proposition \ref{lemma_triangle_inequality_special}]
Without loss of generality we can assume than $\{x_i\}_{i=1}^n$ are in ascending order and let $m\in \{1, \cdots, n\}$ such that $x_m <0$ and $x_{m+1}>0$. Then we can write
\[
|x_1 + \cdots + x_n| = \Big| |x_1+ \cdots + x_m| - |x_{m+1}+ \cdots + x_n|\Big|.
\]
Next, assume that $|x_1+ \cdots + x_m| - |x_{m+1}+ \cdots + x_n| \geq 0$ (the case $<0$ is analogous) then
\[
|x_1 + \cdots + x_n| = |x_1+ \cdots + x_m| - |x_{m+1}+ \cdots + x_n|,
\]
using triangle inequality we get
\begin{align*}
|x_1 + \cdots + x_n| &\leq |x_1+ \cdots + x_m| + \left|x_{m+1}+ \cdots + x_{n-1}\right| - |x_n|\\
&\leq  |x_1|+ \cdots + |x_m| + |x_{m+1}|+ \cdots + |x_{n-1}| - |x_n|\\
& =   |x_1|+ \cdots + |x_{n}| - 2|x_n|.
\end{align*}
This concludes the proof of Proposition \ref{lemma_triangle_inequality_special}.
\end{proof}

Continuation of the proof of Lemma  \ref{lemma_higher_order_sums_muskat}. Because by assumption not all $c_1, \cdots, c_{2k+1}$ have the the same sign, then Proposition \ref{lemma_triangle_inequality_special} tell us
\[
|c_1 + \cdots + c_{2k+1}| \leq |c_1| + \cdots +|c_{2k+1}| - 2 |c_j|,
\]
for some $1\leq j\leq 2k+1$, then we get 
\[
|c_1| + \cdots +|c_{2k+1}| -|c_1 + \cdots + c_{2k+1}| \geq 2 |c_j| \geq 2k_j.
\]
This concludes the proof of Lemma \ref{lemma_higher_order_sums_muskat}.
\end{proof}
\noindent
Continuation of proof Lemma \ref{estimate_higher_order_first_part}. First, apply Lemma \ref{estimate_iterated_convolution_characteristic_functions} part (i) to obtain that in the support of $h_{c_1^j \cdots c_{2k+1}^j}$ given by \eqref{definition_h_higher_order_muskat} we can bound
\begin{equation*}
|c_1^j + \cdots + c_{2k+1}^j - 2k-1|\leq |\xi| \leq |c_1^j + \cdots + c_{2k+1}^j + 2k+1|,
\end{equation*}
next, substituting \eqref{estimate_R_lemma_lower_order_terms_non_periodic} in \eqref{definition_J1_diagonal_terms} and using that $R_{c_1^j\cdots c_{2k+1}^j} = 0$ when all $c_i^j, i=1,\cdots, 2k+1$ have the same sign we obtain the bound
\begin{align*}
\left|\widehat{E(J_{k})}\right| &\leq C\sum\limits_{j = N}^{(1+\delta)N} \sum_{\substack{c_\eta^j \in  \Lambda(k_j) \\ 1\leq \eta \leq 2k+1\\ \text{not all } c_\eta^j\\ \text{ with the same sign}}}\gamma_j^{2k+1}  |\xi| \int_0^t e^{-2\pi (t-\tau)|\xi|} k_j^{2k}\\
& \hspace{1cm} \times e^{-2\pi \tau (|c_1^j|+\cdots+|c_{2k+1}^j|-(2k+1) )} d\tau h_{c_1^j \cdots c_{2k+1}^j}(\xi) \\
&\leq C\sum\limits_{j = N}^{(1+\delta)N} \sum_{\substack{c_\eta^j \in  \Lambda(k_j) \\ 1\leq \eta \leq 2k+1\\ \text{not all } c_\eta^j\\ \text{ with the same sign}}}\gamma_j^{2k+1} |c_1^j + \dots +c_{2k+1}^j+2k+1| k_j^{2k} \\
& \hspace{2cm}\times e^{-2\pi t (|c_1^j + \cdots+c_{2k+1}^j| - 2k-1)} \\
& \hspace{2cm}\times\int_0^t e^{-2\pi \tau (|c_1^j|+ \cdots +|c_{2k+1}^j| - |c_1^j + \cdots +c_{2k+1}^j|)} d\tau \\
& \hspace{2cm}\times  h_{c_1^j \cdots c_{2k+1}^j}(\xi), 
\end{align*}
now by definition of the set $\Lambda(k_j)$ we can bound 
\begin{equation*}
|c_1^j + \cdots +c_{2k+1}^j|\leq (2k+1) (2\ell k_j +M) \leq (2\ell+1)^2 k_j
\end{equation*}
therefore by applying Lemma \ref{lemma_higher_order_sums_muskat} part ii) and iii) we get 
\begin{align}
\left|\widehat{E(J_{k})}\right| &\leq C\sum\limits_{j=N}^{(1+\delta)N} \sum\limits_{\substack{c_\eta^j \in  \Lambda(k_j) \\ 1\leq \eta \leq 2k+1\\ \text{not all } c_\eta^j\\ \text{ with the same sign}}} \gamma_j^{2k+1} (2\ell+1)^2 k_j k_j^{2k}\nonumber \\
&  \hspace{1cm} \times e^{-2\pi t (k_j/2 - 2k-1)} \int_0^t e^{-2\pi \tau (2k_j)} d\tau h_{c_1^j \cdots c_{2k+1}^j}(\xi)\nonumber \\
&\leq C\sum\limits_{j=N}^{(1+\delta)N} \sum\limits_{\substack{c_\eta^j \in  \Lambda(k_j) \\ 1\leq \eta \leq 2k+1\\ \text{not all } c_\eta^j\\ \text{ with the same sign}}} \gamma_j^{2k+1} k_j^{2k} \nonumber \\
&  \hspace{1cm} \times e^{-2\pi t (k_j/2 - 2k-1)} \left(1- e^{-2\pi t (2k_j)}\right) h_{c_1^j \cdots c_{2k+1}^j}(\xi) \nonumber \\
&= C \sum_{j=N}^{(1+\delta)N} \gamma_j^{2k+1} H_j(t) B_j(\xi), \label{estimate_fourier_J2k1}
\end{align}
where 
\begin{equation}\label{bound_Hj}
H_j(t) = k_j^{2k} e^{-2\pi t (k_j/2 - 2k-1)}\left(1-e^{-2\pi t (2k_j)}\right) \leq C  k_j^{2k } e^{-\pi t k_j/2},
\end{equation}
and 
\begin{equation*}
B_j (\xi)= \sum\limits_{\substack{c_\eta^j \in  \Lambda(k_j) \\ 1\leq \eta \leq 2k+1\\ \text{not all } c_\eta^j\\ \text{ with the same sign}}} h_{c_1^j \cdots c_{2k+1}^j}.
\end{equation*}
Using \eqref{estimate_fourier_J2k1} we get the following bound for the $\dot{\mathcal{F}}^{m,p}_q$ norm of $E(J_{k})$
\begin{equation*}
\|E(J_{k})\|_{\dot{\mathcal{F}}^{m,p}_q}^q
\leq C \sum_{r\in\Z} \left(
\int_{C_r} |\xi|^{mp}\left|\sum_{j=N}^{(1+\delta)N} \gamma_j^{2k+1} H_j(t) B_j(\xi)\right|^p\right)^{q/p},
\end{equation*}
now notice that because of condition \eqref{growth_condition_kn_1} for different values of $j$, the terms $B_j(\xi)$ have disjoint support, moreover we can guarantee that on each dyadic annulus at most one among $\{B_j\}_{j=N}^{(1+\delta)N}$ is not identically equal to zero. Let $R_j\subset \Z$ the set of all $k \in \Z$ such that the restriction $\left.B_j\right|_{C_k}$ to the dyadic annulus $C_k=\{\xi\in\R: 2^k \leq |\xi|< 2^{k+1}\}$ is not identically equal to zero, then we can write 
\begin{equation*}
\|E(J_{k})\|_{\dot{\mathcal{F}}^{m,p}_q}^q \leq C \sum_{j=N}^{(1+\delta)N} \sum_{r\in R_j} \left(  (\gamma_j^{2k+1} H_j(t))^p \int_{C_r} |\xi|^{mp} |B_j|^p d\xi \right)^{q/p}. 
\end{equation*}
Next, using Lemma \ref{lemma_higher_order_sums_muskat} we can bound the range of the values on the sum $c_1^j+\cdots + c_{2k+1}^{j}$
\begin{multline*}
2^{\lfloor log_2 k_j \rfloor -1}\leq k_j/2 \leq |c_1^j + \cdots +c_{2k+1}^j|\leq (2k+1) (2\ell k_j +M)\\ \leq (2\ell+1)^2 k_j < 2^{\lfloor log_2 k_j \rfloor+1 + R},
\end{multline*}
where $R\in \N$ is the smallest integer such that $2^R > (2\ell+1)^{2}$. We conclude that the term $B_j(\xi)$ is supported in at most $R+2$ dyadic annulus $C_k$, and we can bound
\begin{equation}\label{bound_I_lower_frequency_nonperiodic}
\|E(J_{k})\|_{\dot{\mathcal{F}}^{m,p}_q}^q \leq   C \sum_{j=N}^{(1+\delta)N} (\gamma_j^{2k+1} H_j(t))^q (R+2) \left(\int_{\R} |\xi|^{mp} |B_j|^p \right)^{q/p}.
\end{equation}
To bound the integral of $|B_j|$ we get from Lemma \ref{estimate_iterated_convolution_characteristic_functions} 
\[
\|B_j\|_{\infty}\leq 2^{2k+1} (\text{number of terms in the sum}) \leq 8^{2k+1},
\]
and 
\begin{align*}
\|B_j\|_1 &\leq \sum\limits_{\substack{c_\eta^j \in  \Lambda(k_j) \\ 1\leq \eta \leq 2k+1\\ \text{not all } c_\eta^j\\ \text{ with the same sign}}} \int_\R |h_{c_1^j \cdots c_{2k+1}^j}(\xi)| d\xi\\
&= \sum\limits_{\substack{c_\eta^j \in  \Lambda(k_j) \\ 1\leq \eta \leq 2k+1\\ \text{not all } c_\eta^j\\ \text{ with the same sign}}} 2^{2k+1}(4k+2)\\
&\leq  8^{2k+1} (4k+2).
\end{align*}
Combining this results we can bound the integral of $|B_j|$ in the following way
\begin{align}
\left(\int_\R |\xi|^{mp} |B_j|^p d\xi\right)^{1/p} &\leq ((2\ell+1)^2 k_j + 2k+1)^{m} \left(\int_\R |B_j|^p d\xi \right)^{1/p} \nonumber\\
&\leq ((2\ell+1)^2 k_j + 2k+1)^{m} \|B_j\|_\infty^{\frac{p-1}{p}} \|B_j\|_1^{1/p} \nonumber\\
&\leq ((2\ell+1)^2 k_j + 2k+1)^{m} 8^{2k+1}(4k+2)^{1/p} \nonumber\\
&= C k_j^{m} , \label{bound_integral_Bj_non_periodic}
\end{align}
for a constant $C=C(m,\ell,k,p)>0$. Finally applying \eqref{bound_Hj} and \eqref{bound_integral_Bj_non_periodic} to \eqref{bound_I_lower_frequency_nonperiodic} we get
\begin{align*}
\|E(J_{k})\|_{\dot{\mathcal{F}}^{m,p}_q} &\leq C \left(\sum_{j=N}^{(1+\delta)N}  \gamma_j^{(2k+1)q}  k_j^{2 k q} e^{-\pi t q k_j/2} k_j^{mq}\right)^{1/q} \\
&\leq \frac{C}{t^{2}} \left(\sum_{j=N}^{(1+\delta)N}  \gamma_j^{(2k+1)q} k_j^{(2k-2+m)q} \right)^{1/q},
\end{align*}
for a constant $C =C(k,\ell,M,p,q) > 0$. This complete the proof of Lemma \ref{estimate_higher_order_first_part}.
\end{proof}
Before proceeding to the proof of Lemma \ref{lemma_higher_frequency_muskat_higher_order} we need an estimate for the term  $R_{c_1^{s_1}\cdots c_{2k+1}^{s_{2k+1}}}$ that will be used several time during the proof.

\begin{lemma}[Estimate $R_{c_1^{s_1}\cdots c_{2k+1}^{s_{2k+1}}}$]\label{lemma_estimate_R_higher_frequency} 
Let $R_{c_1^{s_1}\cdots c_{2k+1}^{s_{2k+1}}}$ as defined by \eqref{definition_R} and 
suppose the following
\begin{itemize}
\item[i)]   $N\leq s_i \leq (1+\delta)N$, for each $i=1, \cdots, 2k+1$,
\item[ii)] $c_i^{s_i}\in \Lambda(k_{s_i})$, $i=1,\cdots, 2k+1$ and 
\item[iii)] not all $c_1^{s_1}$, $\cdots$, $c_{2k+1}^{s_{2k+1}}$ are equal,
\end{itemize}
then given $p\geq 1$, $q\geq 1$ and $0\leq m \leq 1$ we have the estimate '
\begin{equation*}
\left\|E\left(R_{c_1^{s_1}\cdots c_{2k+1}^{s_{2k+1}}}\right)\right\|_{\dot{\mathcal{F}}^{m,p}_{q}}
\leq \frac{C }{t^{m+2} k_N } |k_{s_1} \cdots k_{s_{2k+1}} |^{\frac{2k-1}{2k+1}},
\end{equation*}
where $C= C(m,k,\ell)>0$.
\end{lemma}
\begin{proof}[Proof of Lemma \ref{lemma_estimate_R_higher_frequency}]
First, we note that if all $c_i^{s_i}$, $i=1, \cdots, 2k+1$ have the same sign, then from Lemma  \ref{properties_gamma_higher_dimension} part iii) we have that  
\[
\Gamma_{2k+1}(\xi - \xi_1,\xi_1-\xi_2, \cdots, \xi_{2k}) = 0,
\]
in the region where
\[
\chi_{c_1^{s_1}}(\xi-\xi_1)\chi_{c_2^{s_2}}(\xi_1-\xi_2)\cdots \chi_{c_{2k+1}^{s_{2k+1}}}(\xi_{2k}) \neq 0,
\]
and therefore $\left\|E\left(R_{c_1^{s_1}\cdots c_{2k+1}^{s_{2k+1}}}\right)\right\|_{\dot{\mathcal{F}}^{m,p}_{q}} = 0$. From now on we can assume that not all all $c_i^{s_i}$, $i=1, \cdots, 2k+1$ have the same sign. From Lemma \ref{properties_gamma_higher_dimension} part vi) we know that
\begin{equation*}
\left|\Gamma_{2k+1}(c_1^{s_1},\cdots, c_{2k+1}^{s_{2k+1}})\right| \leq C|c_1^{s_1}\cdots c_{2k+1}^{s_{2k+1}}|^{\frac{2k}{2k+1}}, 
\end{equation*}
applying this estimate to \eqref{expansion_inner_term_higher_order_muskat} we get
\begin{multline*}
\left|R_{c_1^{s_1}\cdots c_{2k+1}^{s_{2k+1}}} \right| \leq
C (|c_1^{s_1}+\cdots +c_{2k+1}^{s_{2k+1}}| + 2k+1)\\
\times|c_1^{s_1}\cdots c_{2k+1}^{s_{2k+1}}|^{\frac{2k}{2k+1}} e^{-2\pi t(|c_1^{s_1}| + \cdots + |c_{2k+1}^{s_{2k+1}}|  -2k-1)} h_{c_1^{s_1}\cdots c_{2k + 1}^{s_{2k+1}}}(\xi),
\end{multline*}
where $h_{c_1^{s_1}\cdots c_{2k + 1}^{s_{2k+1}}}$ is given by \eqref{definition_h_higher_order_muskat}. Next we look at the term $E\left(R_{c_1^{s_1}\cdots c_{2k+1}^{s_{2k+1}}} \right)$
\begin{align}
\left|\mathcal{F}\left(E\left(R_{c_1^{s_1}\cdots c_{2k+1}^{s_{2k+1}}} \right) \right)\right| 
&\leq C \int_0^t e^{-2\pi (t-\tau)|\xi|} (|c_1^{s_1}+\cdots +c_{2k+1}^{s_{2k+1}}| + 2k+1) \nonumber \\
& \hspace{1cm}\times|c_1^{s_1}\cdots c_{2k+1}^{s_{2k+1}}|^{\frac{2k}{2k+1}}\nonumber \\
&\hspace{1cm}	\times e^{-2\pi t(|c_1^{s_1}| + \cdots + |c_{2k+1}^{s_{2k+1}}|  -2k-1)} d\tau\nonumber \\
&\hspace{1cm}	\times h_{c_1^{s_1}\cdots c_{2k + 1}^{s_{2k+1}}}(\xi)\nonumber \\
&\leq C \int_0^t e^{-2\pi (t-\tau)(|c_1^{s_1} +\cdots + c_{2k+1}^{s_{2k+1}}| -2k-1)} \nonumber \\
&\hspace{1cm}\times (|c_1^{s_1}+\cdots +c_{2k+1}^{s_{2k+1}}| + 2k+1)|c_1^{s_1}\cdots c_{2k+1}^{s_{2k+1}}|^{\frac{2k}{2k+1}}\nonumber \\
&\hspace{1cm}\times  e^{-2\pi t( |c_1^{s_1}| + \cdots + |c_{2k+1}^{s_{2k+1}}|-2k-1)} d\tau\nonumber \\
&\hspace{1cm}	\times h_{c_1^{s_1}\cdots c_{2k + 1}^{s_{2k+1}}}(\xi)\nonumber \\
&\leq  C e^{-2\pi t(|c_1^{s_1} +\cdots + c_{2k+1}^{s_{2k+1}}| -2k-1) }\nonumber \\
&\hspace{1cm}\times (|c_1^{s_1}+\cdots +c_{2k+1}^{s_{2k+1}}| + 2k+1)|c_1^{s_1}\cdots c_{2k+1}^{s_{2k+1}}|^{\frac{2k}{2k+1}}\nonumber \\
& \hspace{1cm}\times \int_0^t e^{-2\pi \tau \left(|c_1^{s_1}| + \cdots + |c_{2k+1}^{s_{2k+1}}| - |c_1^{s_1} +\cdots + c_{2k+1}^{s_{2k+1}}| \right)} d\tau\nonumber \\
&\hspace{1cm}	\times h_{c_1^{s_1}\cdots c_{2k + 1}^{s_{2k+1}}}(\xi). \label{bound_R_lemma214}
\end{align}
As in the proof of Lemma \ref{estimate_higher_order_first_part}, a key ingredient for the estimate is to have some control on the size of the term $c_1^{s_1} + \cdots + c_{2k+1}^{s_{2k+1}}$, which is provided by the following generalization of Lemma \ref{lemma_higher_order_sums_muskat}.
\begin{lemma}\label{lemma_size_of_sums_higher_order_muskat}
Let $k \leq \ell$, $c_i \in \cup_{j =N}^{(1+\delta)N}\{\pm k_j, \pm (2k_j + M)\}$ for $i=1, \cdots, 2k+1$ and suppose that $\{k_j\}_{i=1}^{\infty}$ satisfy \eqref{growth_condition_kn_1} and \eqref{growth_condition_kn_3}. Then we have the following

\begin{itemize}
\item[i)] If $c_1+ \cdots +c_{2k+1} \neq \pm M$ then 
\[
|c_1 +\cdots + c_{2k+1}| \geq |k_{i_1}|/4,
\] 
for some $N\leq i_1\leq (1+\delta)N$,
\item[ii)] If not all $c_i, i=1, \cdots, 2k+1$ have the same sign then
\[
|c_1| + \cdots + |c_{2k+1}| - |c_1 +\cdots + c_{2k+1}| \geq 2|k_{i_2}|,
\]
for some $N\leq i_2\leq (1+\delta)N$,
\item[iii)]  $\max\left\{|c_1+ \cdots + c_{2k+1}|,|c_1| + \cdots + |c_{2k+1}| - |c_1+\cdots +c_{2k+1}|\right    \}$\\ $\geq \max_j |c_{j}|/2$.
\end{itemize}
\end{lemma}
\begin{proof}[Proof of Lemma \ref{lemma_size_of_sums_higher_order_muskat}]
For part $i)$ we write $c_{j} = a_{j} k_{i_{j}} + \varepsilon_{j} M$, where $a_j \in \{\pm 1, \pm 2\ell\}$, $\varepsilon\in \{-1,0,1\}$. Then we get
\begin{equation}\label{sum_ci_lemma}
c_1 + \cdots + c_{2k+1}= \sum_{j=N}^{(1+\delta)N}b_j k_j 
 + (\varepsilon_1+ \cdots + \varepsilon_{2k+1}) M ,
\end{equation}
where $b_j = \sum\{a_n: N\leq n \leq (1+\delta)N ,k_{i_n} = k_j\}$. From condition \eqref{growth_condition_kn_1}, we get that
$\sum_{j=N}^i b_j k_j \leq \frac{1}{2} k_{i+1}$ and therefore in order for  $\sum_{j=N}^{(1+\delta)N}b_j k_j $ to vanish we need that each $b_j$ is identically equal to zero. 

To see that not all $b_j$, $N\leq j \leq (1+\delta)N$  can vanish simultaneously, we first notice that if for some $n\in\{1,\cdots, 2k+1\}$ we have $a_n = \pm 2\ell$, then by Lemma \ref{lemma_higher_order_sums_muskat} part i) if the corresponding $b_{i_n}$ is equal to zero, that would imply that $k_{i_j} = k_{i_n}$ for all $j\in\{1,\cdots,2k+1\}$ and that $|c_1 + \cdots + c_{2k+1}| = M$, which is a contradiction with our assumption. On the other, if for all $n\in \{1, \cdots, 2k+1\}$ we have $a_n = \pm 1$, then because  $\sum_{j=N}^{(1+\delta)N} b_j =  \sum_{n=1}^{2k+1} a_n$  is an odd number, at least one $b_j$, $N\leq j \leq (1+\delta)N$ is not identically equal to zero. 

Next because we know that at least one $b_j$, $N\leq j \leq (1+\delta)N$ is non zero, we can take 
$\hat{i}$ to be the largest of such $j$, then we can bound
\begin{equation*}
|a_1 k_{i_1} + \cdots + a_{2k+1}k_{i_{2k+1}}| \geq |b_{\hat{i}} k_{\hat{i}}| - \sum_{j=N}^{\hat{i}-1}|b_j| k_j\geq \frac{1}{2}k_{\hat{i}}.
\end{equation*}
Next, because in \eqref{sum_ci_lemma} we can bound $|\varepsilon_1+ \cdots + \varepsilon_{2k+1}| \leq {2k+1}$, we obtain using assumption \eqref{growth_condition_kn_1} that 
\begin{equation*}
|c_1 + \cdots + c_{2k+1}| \geq \frac{1}{2} k_{\hat{i}} - \frac{1}{2}k_0 \geq \frac{1}{4} k_{\hat{i}}.
\end{equation*}
This concludes the proof of part i). For part ii) the key argument is that because not all the $c_i$ have the same sign, then by Lemma \ref{lemma_triangle_inequality_special} we can bound
\begin{equation*}
|c_1 + \cdots + c_{2k+1}| \leq |c_1| + \cdots +|c_{2k+1}| - 2|c_{m}|,
\end{equation*}
for some $m\in \{1, \cdots, 2k+1\}$, therefore we conclude that
\begin{equation*}
|c_1| + \cdots +|c_{2k+1}| - |c_1 + \cdots + c_{2k+1}| \geq 2|c_{m}| \geq 2 k_{i_m}.
\end{equation*}
This completes the proof of part ii). Part $iii)$ comes from the observation that 
\begin{multline*}
\left(|c_1+ \cdots + c_{2k+1}|\right) + \left(|c_1| + \cdots + |c_{2k+1}| - |c_1+\cdots+ c_{2k+1}|\right)\\
 =|c_1| + \cdots + |c_{2k+1}| \geq \max_j |c_{j}|,
\end{multline*}
and because both terms are non-negative we get the result. This concludes the proof of Lemma \ref{lemma_size_of_sums_higher_order_muskat}.
\end{proof}
Continuation of proof of Lemma \ref{lemma_estimate_R_higher_frequency}. Using that  $x^{m+2} e^{-x}\leq C$
\begin{align*}
\left|\mathcal{F}\left(E(R_{c_1^{s_1}\cdots c_{2k+1}^{s_{2k+1}}} ) \right)\right| &\leq  \frac{C}{t^{m+2}} |c_1^{s_1}\cdots c_{2k+1}^{s_{2k+1}}|^{\frac{2k}{2k+1}} (|c_1^{s_1}+\cdots +c_{2k+1}^{s_{2k+1}}| + 2k+1)\\
&\hspace{1cm}\times \frac{1}{(|c_1^{s_1} +\cdots + c_{2k+1}^{s_{2k+1}}| -2k-1)^{m + 2}}\\
&\hspace{1cm}\times e^{-\pi t(|c_1^{s_1} +\cdots + c_{2k+1}^{s_{2k+1}}| -2k-1) } \\ 
& \hspace{1cm}\times \frac{1}{\left(|c_1^{s_1}| + \cdots + |c_{2k+1}^{s_{2k+1}}| - |c_1^{s_1} +\cdots + c_{2k+1}^{s_{2k+1}}| \right)} \\
& \hspace{1cm}\times \left(1- e^{-2\pi \tau \left(|c_1^{s_1}| + \cdots + |c_{2k+1}^{s_{2k+1}}| - |c_1^{s_1} +\cdots + c_{2k+1}^{s_{2k+1}}|  \right)} \right)\\
&\hspace{1cm}	\times h_{c_1^{s_1}\cdots c_{2k + 1}^{s_{2k+1}}}(\xi),
\end{align*}
to bound the product in the denominator we use the following inequality
\begin{multline*}
|c_1^{s_1} +\cdots + c_{2k+1}^{s_{2k+1}}| \times \left( |c_1^{s_1}| + \cdots + |c_{2k+1}^{s_{2k+1}}| - |c_1^{s_1} +\cdots + c_{2k+1}^{s_{2k+1}}|\right)\\ \geq \frac{1}{4}k_N \frac{1}{2} \max_i |c_{i}^{s_{i}}|.
\end{multline*}
To obtain this bound, because of Lemma \ref{lemma_size_of_sums_higher_order_muskat} part iii) we know that at least one factors can be bounded below by $\frac{1}{2} \max_i |c_{i}^{s_{i}}|$ and for the other one we apply Lemma \ref{lemma_size_of_sums_higher_order_muskat} part i) to get the lower bound of $k_N/4$. Applying this bound to our previous computation we obtain
\begin{multline*}
\left|\mathcal{F}\left(E(R_{c_1^{s_1}\cdots c_{2k+1}^{s_{2k+1}}} ) \right)\right| 
\leq \frac{C}{t^{m+2} k_N} \frac{|c_1^{s_1}\cdots c_{2k+1}^{s_{2k+1}}|^{\frac{2k}{2k+1}}}{\max_i |c_{i}^{s_{i}}|} \\
\times \frac{e^{-\pi t|c_1^{s_1} +\cdots + c_{2k+1}^{s_{2k+1}}|/2}}{(|c_1^{s_1} +\cdots + c_{2k+1}^{s_{2k+1}}| -2k-1)^{m}}	 h_{c_1^{s_1}\cdots c_{2k + 1}^{s_{2k+1}}}(\xi).
\end{multline*}
Now notice that $\max_i |c_{i}^{s_{i}}| \geq \frac{c_{j}^{s_{j}}}{(2\ell +1)}$ for all $j$, therefore we can bound
\begin{equation*}
\frac{ |c_1^{s_1}\cdots c_{2k+1}^{s_{2k+1}}|^{\frac{2k}{2k+1}}}{\max_i |c_{i}^{s_{i}}|}
\leq (2\ell+1) |c_1^{s_1}\cdots c_{2k+1}^{s_{2k+1}}|^{\frac{2k-1}{2k+1}},
\end{equation*}
combining this with our previous estimate
\begin{multline}\label{first_estimate_general_term_higher_freq_higher_order_muskat}
\left\|E\left(R_{c_1^{s_1}\cdots c_{2k+1}^{s_{2k+1}}}\right)\right\|_{\dot{\mathcal{F}}^{m,p}_{q}}
\leq \frac{C }{t^{m+2} k_N } |c_1^{s_1}\cdots c_{2k+1}^{s_{2k+1}}|^{\frac{2k-1}{2k+1}} \|\mathcal{F}^{-1}(h_{c_1^{s_1}\cdots c_{2k + 1}^{s_{2k+1}}})\|_{\dot{\mathcal{F}}^{m,p}_{q}} \\
\times \frac{e^{-\pi t|c_1^{s_1} +\cdots + c_{2k+1}^{s_{2k+1}}|/2}}{(|c_1^{s_1} +\cdots + c_{2k+1}^{s_{2k+1}}| -2k-1)^{m}},
\end{multline}
for a constant $C = C(k,p,q,\ell) >0$. Next by Lemma \ref{estimate_iterated_convolution_characteristic_functions} part (ii) we can bound the norm of $h_{c_1^{s_1}\cdots c_{2k + 1}^{s_{2k+1}}}(\xi)$ by
\begin{equation*}
\|\mathcal{F}^{-1}(h_{c_1^{s_1}\cdots c_{2k + 1}^{s_{2k+1}}}(\xi))\|_{\dot{\mathcal{F}}^{m,p}_{q}} 
\leq  C |c_1 + \cdots + c_{2k+1} + 2k+1|^{m},
\end{equation*}
for a constant $C = C(k,p) > 0$. Applying this estimate to \eqref{first_estimate_general_term_higher_freq_higher_order_muskat} and using that by Lemma \ref{lemma_size_of_sums_higher_order_muskat} part i),  $|c_1^{s_1}+\cdots + c_{2k+1}^{s_{2k+1}}| \geq k_N/4$ we obtain
\begin{align*}
\left\|E\left(R_{c_1^{s_1}\cdots c_{2k+1}^{s_{2k+1}}}\right)\right\|_{\dot{\mathcal{F}}^{m,p}_{q}}
&\leq C\frac{|c_1^{s_1}\cdots c_{2k+1}^{s_{2k+1}}|^{\frac{2k-1}{2k+1}}  }{t^{m+2} k_N } \frac{(|c_1^{s_1} +\cdots + c_{2k+1}^{s_{2k+1}}| +2k+1)^{m}}{(|c_1^{s_1} +\cdots + c_{2k+1}^{s_{2k+1}}| -2k-1)^{m}}\\
&\hspace{1cm}	\times e^{-\pi t|c_1^{s_1} +\cdots + c_{2k+1}^{s_{2k+1}}|/2}\\
&\leq  \frac{C }{t^{m+2} k_N } |c_1^{s_1}\cdots c_{2k+1}^{s_{2k+1}}|^{\frac{2k-1}{2k+1}}\\
&\leq  \frac{C }{t^{m+2} k_N } (2\ell+1)^{\frac{2k-1}{2k+1}}|k_{s_1} \cdots k_{s_{2k+1}} |^{\frac{2k-1}{2k+1}},
\end{align*}
this concludes the proof of Lemma \ref{lemma_estimate_R_higher_frequency}. 
\end{proof}

\begin{proof}[Proof of Lemma \ref{lemma_higher_frequency_muskat_higher_order}]
By applying Lemma \ref{lemma_estimate_R_higher_frequency} to equation \eqref{definition_HF_higher_order_muskat} we get

\begin{align*}
\|E(HF)\|_{\dot{\mathcal{F}}^{m,p}_{q}} &= \left\|\sum\limits_{\substack{N\leq s_i \leq (1+\delta)N\\ 1\leq i \leq 2k+1\\ \text{not all equal }}} \sum\limits_{\substack{c_i^{s_i}\in \Lambda(k_{s_i}) \\ 1\leq i \leq 2k+1}} \gamma_{s_1} \cdots \gamma_{s_{2k+1}} E(R_{c_1^{s_1}\cdots c_{2k+1}^{s_{2k+1}}})\right\|_{\dot{\mathcal{F}}^{m,p}_{q}}\\
&\leq \sum\limits_{\substack{N\leq s_i \leq (1+\delta)N\\ 1\leq i \leq 2k+1\\ \text{not all equal }}} \sum\limits_{\substack{c_i^{s_i}\in \Lambda(k_{s_i}) \\ 1\leq i \leq 2k+1}} \gamma_{s_1} \cdots \gamma_{s_{2k+1}} \left\|E(R_{c_1^{s_1}\cdots c_{2k+1}^{s_{2k+1}}})\right\|_{\dot{\mathcal{F}}^{m,p}_{q}}\\
&\leq \frac{C }{t^{m+2}  k_N} \sum\limits_{\substack{N\leq s_i \leq (1+\delta)N\\ 1\leq i \leq 2k+1\\ \text{not all equal }}} \sum\limits_{\substack{c_i^{s_i}\in \Lambda(k_{s_i}) \\ 1\leq i \leq 2k+1}} \gamma_{s_1} \cdots \gamma_{s_{2k+1}} |k_{s_1} \cdots k_{s_{2k+1}} |^{\frac{2k-1}{2k+1}}\\
&\leq \frac{C}{t^{m+2}  k_N} \left(\sum_{j=N}^{(1+\delta)N} \gamma_j  k_j^{\frac{2k-1}{2k+1}}\right)^{2k+1},
\end{align*}
for a constant $C= C(k,p,q,\ell) > 0$. This completes the proof of Lemma \ref{lemma_higher_frequency_muskat_higher_order}.
\end{proof}
\noindent
Continuation of proof of Lemma \ref{lemma_estimate_lower_order_terms_higher_order_muskat}. Using the estimates given by Lemmas \ref{estimate_higher_order_first_part} and \ref{lemma_higher_frequency_muskat_higher_order} we get that
\begin{align*}
\|f_{k} \|_{\dot{\mathcal{F}}^{m,p}_q} &\leq \left\|E(J_{k})\right\|_{\dot{\mathcal{F}}^{m,p}_q} +\left\| E(HF) \right\|_{\dot{\mathcal{F}}^{m,p}_q}\\
&\leq  \frac{C}{t^{2}} \left(\sum_{j=N}^{(1+\delta)N}  \gamma_j^{(2k+1)q}  k_j^{(2k-2+m)q} \right)^{1/q}\\
&\hspace{1cm}+ \frac{C}{t^{m+2}  k_N}  \left(\sum_{j=N}^{(1+\delta)N} \gamma_j  k_j^{\frac{2k-1}{2k+1}}\right)^{2k+1},
\end{align*}
for a constant $C=C(m,q,\ell,k)>0$. This concludes the proof of Lemma \ref{lemma_estimate_lower_order_terms_higher_order_muskat}.
\end{proof}

\begin{proof}[Proof of Lemma \ref{lemma_estimate_main_term_higher_order_muskat} ]
As in the proof of Lemma \ref{lemma_estimate_lower_order_terms_higher_order_muskat}, from equation \eqref{full_expansion_f2k1} we can write the following decomposition of $f_{\ell}$
\begin{equation}\label{decomposition_main_term_higher_order_muskat}
f_{\ell} = \frac{(-1)^{\ell}}{2\ell+1} \Big(E(J_{\ell}) + E(HF_1) + E(HF_2) \Big),
\end{equation}
where
\begin{equation}\label{definition_J1_lemma_main_term_non_periodic}
J_{\ell}(\xi,t) = \sum_{j=N}^{(1+\delta)N} \sum_{\substack{c_\eta^{j}\in \Lambda(k_j)\\ 1\leq \eta \leq 2\ell+1 \\ c^j_1+\cdots + c^j_{2\ell+1} = \pm M }} \gamma_j^{2\ell+1} R_{c^j_1\cdots c^j_{2\ell+1}}(\xi,t),
\end{equation}
\begin{equation}\label{definition_HF1_lemma_main_term_non_periodic}
HF_1(\xi,t) = \sum_{j=N}^{(1+\delta)N} \sum_{\substack{c_\eta^{j}\in \Lambda(k_j)\\ 1\leq \eta \leq 2\ell+1 \\ c^j_1+\cdots + c^j_{2\ell+1} \neq \pm M }} \gamma_j^{2\ell+1} R_{c_1^j\cdots c_{2\ell+1}^j}(\xi,t),
\end{equation}
and 
\begin{equation}\label{definition_HF2_lemma_main_term_non_periodic}
HF_2(\xi,t) = \sum\limits_{\substack{N\leq s_i \leq (1+\delta)N\\ 1\leq i \leq 2\ell+1\\ \text{not all equal }}} \sum_{\substack{c_\eta^{s_\eta} \in \Lambda(k_{s_\eta})\\ 1\leq \eta \leq 2\ell+1}} \gamma_{s_1} \cdots \gamma_{s_{2\ell + 1}}  R_{c_1^{s_1}\cdots c_{2\ell+1}^{s_{2\ell+1}}}(\xi,t) ,
\end{equation}
where $R_{c_1^{s_1}\cdots c_{2\ell+1}^{s_{2\ell+1}}}$ is defined as \eqref{definition_R}. To estimate these terms we use the following Lemmas.

\begin{lemma}\label{lower_bound_main_term_muskat_higher_order}
Let $M> 2\ell+2$, $tM \leq 1$, $t k_N > 1$, and $J_{\ell}$ as given by \eqref{definition_J1_lemma_main_term_non_periodic} then
\begin{equation*}
\|E(J_{\ell})\|_{\dot{\mathcal{F}}^{m,p}_q}\geq  C \sum_{j=N}^{(1+\delta)N}  \gamma^{2\ell+1} k_j^{2\ell-1},
\end{equation*}
where $C = C(p, q, m, M,  \ell) >0$.
\end{lemma}
\begin{lemma}\label{upper_bound_high_freq_main_term_higher_order_muskat}
Let $HF_1$ and $HF_2$ as defined by \eqref{definition_HF1_lemma_main_term_non_periodic} and \eqref{definition_HF2_lemma_main_term_non_periodic}, then
\begin{equation*}
\|E(HF_1)\|_{\dot{\mathcal{F}}^{m,p}_q}  \leq  \frac{C}{t^{m+2}  k_N} \sum_{j=N}^{(1+\delta)N} \gamma_j^{2\ell+1}  k_j^{2\ell-1},
\end{equation*}
and
\begin{equation*}
\|E(HF_2)\|_{\dot{\mathcal{F}}^{m,p}_q} \leq  \frac{C}{t^{m+2}  k_N} \left(\sum_{j=N}^{(1+\delta)N} \gamma_j  k_j^{\frac{2\ell-1}{2\ell+1}}\right)^{2\ell+1},
\end{equation*}
where $C = C(p, q, m, M,  \ell) >0$.
\end{lemma}

\begin{proof}[Proof of Lemma \ref{lower_bound_main_term_muskat_higher_order}]

The key element for this proof is a lower bound of $R_{c_1^j\cdots c_{2k+1}^j}$, for this purpose we need to estimate the value of 
\begin{equation*}
\Gamma_{2\ell + 1}(-k_j,\cdots ,-k_j, 2\ell k_j + M),
\end{equation*}
this can be done using the integral formula for $\Gamma_{2\ell+1}$ given by definition \ref{definition_gamma_multidimensional}
\begin{align*}
L &= \Gamma_{2\ell + 1}(-k_j,\cdots ,-k_j, 2\ell k_j + M) \\
&= i\int_\R \frac{(1-e^{2\pi i k_j \alpha})^{2\ell} (1-e^{-2\pi i (2k_j\ell + M)\alpha})}{\alpha^{2\ell+1}}d\alpha\\
&= i\int_\R \frac{(1-e^{2\pi i k_j \alpha})^{2\ell} (1-e^{-2\pi i (2\ell k_j )\alpha})}{\alpha^{2\ell+1}}d\alpha \\
&\hspace{1cm}+ i\int_\R \frac{(1-e^{2\pi i k_j \alpha})^{2\ell} (e^{-2\pi i (2\ell k_j)\alpha}-e^{-2\pi i (2\ell k_j + M)\alpha})}{\alpha^{2\ell+1}}d\alpha\\
&= I_1 + I_2.
\end{align*}
For $I_1$ we have
\begin{align*}
I_1 &=i\int_\R \frac{(1-e^{2\pi i k_j \alpha})^{2\ell} (1-e^{-2\pi i (2 k_j\ell)\alpha})}{\alpha^{2\ell+1}}d\alpha \\
&=i\int_\R \frac{(e^{-i \pi k_j \alpha}-e^{\pi i k_j \alpha})^{2\ell} (e^{\pi i (2\ell k_j) \alpha}-e^{-\pi i (2 \ell k_j) \alpha})}{\alpha^{2\ell+1}}d\alpha\\
&= i(2i)^{2\ell+1} k_j^{2\ell} \pi^{2\ell}\int_\R \frac{\sin(\beta)^{2\ell} \sin(2\ell\beta)}{\beta^{2\ell+1}}d\beta\\
&=  (-1)^{\ell+1} (2\pi)^{2\ell+1} k_j^{2\ell} .
\end{align*}
For the second term we use the following
\begin{align*}
I_2 &= i\int_\R \frac{(1-e^{2\pi i k_j \alpha})^{2\ell} (e^{-2\pi i (2\ell k_j)\alpha}-e^{-2\pi i (2\ell k_j + M)\alpha})}{\alpha^{2\ell+1}}d\alpha\\
&= i\int_\R \frac{(e^{-2\pi i k_j \alpha }-1)^{2\ell} (1-e^{-2\pi i M\alpha})}{\alpha^{2\ell+1}}d\alpha\\
&= \Gamma_{2\ell + 1}(k_j , \cdots , k_j, M)\\
&= k_j^{2\ell-1} |k_j| \Gamma_{2\ell+1}(1,\cdots ,1,M/k_j),
\end{align*}
then by applying Lemma \ref{properties_gamma_higher_dimension} part v) we conclude
\begin{equation*}
|I_2| \leq |k_j|^{2\ell}  (2\pi)^{2\ell+1} \frac{M}{|k_j|}
= |k_j|^{2\ell-1}  (2\pi)^{2\ell+1} M.
\end{equation*}
By putting the estimates for $I_1$ and $I_2$ together we get
\begin{equation}\label{computation_gamma_especial_non_periodic}
\Gamma_{2\ell + 1}(-k_j,\cdots ,-k_j, 2\ell k_j + M) = (-1)^{\ell+1} (2\pi)^{2\ell+1} k_j^{2\ell} + O(k_j^{2\ell-1}),
\end{equation}
next, applying \eqref{computation_gamma_especial_non_periodic} to \eqref{expansion_inner_term_higher_order_muskat} we get the estimate
\begin{equation*}
(-1)^{\ell+1}R_{c_1^j\cdots c_{2\ell+1}^j } \geq C |\xi| k_j^{2\ell} e^{-2\pi t(4\ell k_j + M + 2\ell+1)} h_{c_1^j \cdots c_{2\ell+1}^j}(\xi),
\end{equation*}
where $h_{c_1^j \cdots c_{2\ell+1}^j}(\xi)$ is given by \eqref{definition_h_higher_order_muskat} and a constant $C= C(M,\ell)>0$. Next, we use that by Lemma \ref{lemma_higher_order_sums_muskat} part i) the only way of getting $c_1^j+\cdots+ c_{2k+1}^j = \pm M$ with $c_i^j\in\Lambda(k_j)$, $i=1,\cdots, 2\ell+1$ are the tuples such that exactly one of the elements is equal to $\pm (2\ell k_j +M)$ and all the others equal to $\mp k_j$, and therefore this estimate apply to all terms in the sum \eqref{definition_J1_lemma_main_term_non_periodic}. Finally, using the lower bound from Lemma \ref{estimate_iterated_convolution_characteristic_functions} we get for $E(J_{\ell})$
\begin{align*}
\|E(J_{\ell})\|_{\dot{\mathcal{F}}^{m,p}_q}  &\geq C(M-2\ell-1)^{1+m}( 4\ell+2) \\
& \hspace{1cm}\times \sum_{j=N}^{(1+\delta)N}  \gamma_{j}^{2\ell+1} k_j^{2\ell}  e^{-2\pi t(M+2\ell+1)}\frac{(1-e^{-2\pi t (4 \ell k_j)})}{2\pi (4 \ell k_j)}\\
&\geq C \sum_{j=N}^{(1+\delta)N} \gamma_{j}^{2\ell+1} k_j^{2\ell-1} ,
\end{align*}
for a constant $C= C(M,\ell,m)>0$. And because $M > 2\ell+2$, $t M \leq 1$, $t k_j > 1$, this concludes the proof of Lemma \ref{lower_bound_main_term_muskat_higher_order}.

\end{proof}

\begin{proof}[Proof of Lemma \ref{upper_bound_high_freq_main_term_higher_order_muskat}]
As in the proof on Lemma \ref{lemma_higher_frequency_muskat_higher_order}, by applying Lemma  \ref{lemma_estimate_R_higher_frequency} to \eqref{definition_HF1_lemma_main_term_non_periodic} and \eqref{definition_HF2_lemma_main_term_non_periodic} we get
\begin{equation*}
\|E(HF_1)\|_{\dot{\mathcal{F}}^{m,p}_q} \leq   \frac{C}{t^{m+2}  k_N} \sum_{j=N}^{(1+\delta)N} \gamma_j^{2k+1}  k_j^{2k-1},
\end{equation*}
and
\begin{equation*}
\|E(HF_2)\|_{\dot{\mathcal{F}}^{m,p}_q}  \leq  \frac{C}{t^{m+2}  k_N} \left(\sum_{j=N}^{(1+\delta)N} \gamma_j  k_j^{\frac{2k-1}{2k+1}}\right)^{2k+1},
\end{equation*}
for a constant $C = C(m, p, q, \ell) > 0$, this complete the proof of Lemma \ref{upper_bound_high_freq_main_term_higher_order_muskat}
\end{proof}

Continuation of proof Lemma \ref{lemma_estimate_main_term_higher_order_muskat}. From the decomposition given by equation \eqref{decomposition_main_term_higher_order_muskat} we can bound the $\dot{\mathcal{F}}^{m,p}_q$ norm of $f_{\ell}$ by
\begin{equation*}
\|f_{\ell}\|_{\dot{\mathcal{F}}^{m,p}_q} \geq \|E(J_{\ell})\|_{\dot{\mathcal{F}}^{m,p}_q} - \|E(HF_1)\|_{\dot{\mathcal{F}}^{m,p}_q} -  \|E(HF_2)\|_{\dot{\mathcal{F}}^{m,p}_q}, 
\end{equation*}
then by Lemmas \ref{lower_bound_main_term_muskat_higher_order}  and \ref{upper_bound_high_freq_main_term_higher_order_muskat} we get the estimate
\begin{multline*}
\|f_{\ell}\|_{\dot{\mathcal{F}}^{m,p}_q} \geq C \sum_{j=N}^{(1+\delta)N} \gamma_j^{2\ell +1} k_j^{2\ell-1}\\ 
-  \frac{C}{t^{m+2} k_N} \sum_{j=N}^{(1+\delta)N} \gamma_j^{2\ell+1}  k_j^{2\ell-1}
-  \frac{C}{t^{m+2}  k_N} \left(\sum_{j=N}^{(1+\delta)N} \gamma_j  k_j^{\frac{2\ell-1}{2\ell+1}}\right)^{2\ell+1},
\end{multline*}
which concludes the proof of Lemma \ref{lemma_estimate_main_term_higher_order_muskat}
\end{proof}

\end{document}